\documentclass[reqno]{amsart}
\usepackage[hidelinks]{hyperref}
\usepackage{enumitem}
\usepackage{tikz}
\usepackage{doi}
\usepackage[utf8]{inputenc}
\usepackage{amsmath}
\usepackage{amssymb} 
\usepackage[numbers]{natbib}
\usepackage{prettyref}

\def\pair#1#2{{#1}\mid {#2}}

\theoremstyle{plain}
\makeatletter
\newtheorem*{thm@P}{Theorem \pipp@}

\makeatother

\numberwithin{equation}{section}

\newtheorem{thm}{Theorem}[section]
\newtheorem*{thm*}{Theorem}
\newrefformat{thm}{Theorem~\ref{#1}} 
\makeatletter 
\let\old@newtheorem\newtheorem 
\renewcommand{\newtheorem}[2]{\old@newtheorem{#1}[thm]{#2} 
	\newrefformat{#1}{#2~\ref{##1}}} 
\makeatother 

\newtheorem{prop}{Proposition}
\newtheorem{cor}{Corollary}

\newtheorem{lem}{Lemma}

\theoremstyle{definition}

\newtheorem{defn}{Definition}

\newtheorem{exa}{Example}

\theoremstyle{remark}
\newtheorem{rem}{Remark}

\newcommand{\Q}{\mathbb Q}
\newcommand{\R}{\mathbb R}
\newcommand{\N}{\mathbb N}

\newcommand\ntrunc{\mathrel{\blacktriangleleft}}
\newcommand\ntrunceq{\mathrel{\ooalign{{\raise-1ex\hbox{$\relbar$}}\cr\raise.1ex\hbox{$\ntrunc$}}}}
\DeclareMathOperator{\birth}{birthday}
\DeclareMathOperator{\nr}{NR}

\newcommand{\no}{\mathbf {No}}
\newcommand{\Z}{\mathbb Z}
\newcommand{\n}{\mathfrak n}
\newcommand{\m}{\mathfrak m}
\newcommand{\x}{\mathbf{x}}
\newcommand{\nin}{\not\in}

\newcommand{\rest}{\! \restriction \!}

\DeclareMathOperator{\ind}{index}

\def\eps{\varepsilon}

\def\oxp#1{{\omega}^{#1}}
\def\orexp#1{{\omega}^{\cdot #1}}
\DeclareMathOperator{\suchthat}{\,:\,}

\def\on{\mathbf{On}}
\def\M{\mathfrak{M}}

\def\T{\mathbb{T}}

\title{Surreal numbers, exponentiation and derivations}

\author{Alessandro Berarducci}
\address{Alessandro Berarducci, Università di Pisa, Dipartimento di Matematica, Largo Bruno Pontecorvo 5, 56127 Pisa, Italy}
\date{15 August 2020}
\thanks{Partially supported by the Italian research project PRIN 2017, ``Mathematical logic: models, sets, computability'', Prot. 2017NWTM8RPRIN.}
\begin{document}

	\begin{abstract}
		We give a presentation of Conway's surreal numbers focusing on the connections with transseries and Hardy fields and trying to simplify when possible the existing treatments. 
		\end{abstract}
\maketitle
\tableofcontents

	\section{Introduction} 
		Conway's field $\no$ of surreal numbers \citep{Conway76} includes  both  the field of real numbers $\R$ and the class $\on$ of all ordinal numbers. The surreals originally emerged as a subclass of the larger class of {\em games}, comprising for instance the game of ``Go'' and similar combinatorial games. In this paper we shall however be interested in
		the more recent connections with transseries and Hardy fields.
		
		The field $\T$ of transseries \cite{Dries1997, DriesMM2001} (see \cite{Kuhlmann2000,Kuhlmann2012d} for a different variant) is an extension of the field of Puiseux series that
		plays an important role in Ecalle's positive solution of the problem of Dulac \cite{Dulac1923, Ecalle1992}: the finiteness of limit cycles in polynomial planar vector fields. Unlike the Puiseux series, the transseries are closed under formal integration and admit an exponential and a logarithmic function. It is possible to consider $\T$ as a universal domain for the existence of solutions of an important class of formal differential equations of non-oscillatory nature. This is made precise in \cite{Aschenbrenner2017}, where the first-order theory of $\T$ is shown to be recursively axiomatisable and model complete. 
		
		A first connection between $\T$ and $\no$ comes from the fact that $\no$ admits a representation in terms of generalised series \cite{Conway76} and has an exponential function $\exp:\no\to \no$ extending the real exponential function \cite{Gonshor1986}. This makes it possible to interpret surreal numbers as asymptotic expansions, as in \cite{Berarducci2019a}. We can then see $\T$ as a substructure of $\no$, with the ordinal $\omega$ playing the role of the formal variable, and introduce a strongly additive surjective derivation $\partial:\no\to \no$ compatible with the exponential function and extending the derivation of $\T$ \cite{Berarducci2018}. It turns out that $\T$ is an elementary substructure of $\no$ both as a differential field \cite{Aschenbrenner2019}, and as an exponential field \cite{DriesE2001} (based on \cite{Ressayre1993,Dries1994}). Very recently Kaplan has announced an axiomatisation and a model completeness result in the language with both $\partial$ and $\exp$, thus showing that $\T$ is an elementary substructure of $\no$ in this language \cite{Kaplan2020}. 
			Another connection with transseries comes from the work of \citet*{Costin2015}, where it is shown that Ecalle-Borel transseriable functions extend naturally to $\no$. 
			
			A remarkable feature of $\no$, not shared by the field of transseries, is its universality. Every divisible ordered abelian group is isomorphic to an initial subgroup of $\no$, and every real-closed field is isomorphic to an initial subfield of $\no$ \cite{Ehrlich2001a}. In a similar spirit it can be  shown that every Hardy field can be embedded in $\no$ as a differential field \cite{Aschenbrenner2019}.
			 
			In this paper we shall give a detailed presentation of the rich structure on $\no$, including the exponential and differential structure, and describe some of the connections with transseries and Hardy fields. This paper can be read as a supplement to \cite{Mantova2017a}. 
%
%
	
	\section{Simplicity and order}
Let $\on$ be the class of von Neumann {\bf ordinals}. We recall that an ordinal coincides with the set of all smaller ordinals. \citet{Conway76} defined the surreals as a subclass of the class of ``games''. He then showed that it is possible to represent each surreal as a transfinite binary sequence, called its {\bf sign-expansion}.  
Following \citet{Gonshor1986} we define the surreal numbers directly as sign-expansions. More precisely, a {\bf surreal number} is a function $x: \alpha \to \{-,+\}$ from some ordinal $\alpha$ to $\{-,+\}$. If $x$ is as above, we call $\alpha$ the  {\bf birthday} 
of $x$ and we write $\alpha= \birth(x)$. We also say that $x$ is born on day $\alpha$. 

	The class $\no$ of all surreal numbers has a natural structure of a complete binary tree whose nodes are the sign-expansions. 
	The {\bf ancestors} of a surreal $y:\alpha \to \{-1,1\}$ are the restrictions of $y$ to a smaller ordinal $\beta < \alpha$. If $x$ is an ancestor of $y$, we say that $x$ is {\bf simpler} than $y$, or that $y$ is a {\bf descendant} of $x$. If $x$ is simpler than $y$, then clearly $\birth(x) < \birth(y)$. The empty sequence, born on day $0$, is the root of the tree and coincides with the simplest surreal number. Each surreal $x:\alpha\to \{-,+\}$ has two successors: a {\bf left-successor} ``$x-$'' obtained by appending the sign ``$-$'' at the end of $x$, and a {\bf right-successor} ``$x+$'' obtained by appending the sign ``$+$'' at the end of $x$. These are the immediate descendants of $x$. Each descendant of $x$ is either an immediate descendant, or a descendant of an immediate descendant. 
	
	We can now introduce a {\bf total order} $<$ on $\no$ which essentially consists in projecting the tree on a horizontal axis parallel to a line of ``siblings'' in the tree. More precisely, a surreal $x$ is bigger than its left-successor and its descendants, and smaller than its right-successor and its descendants. Thus $$x- < x < x+$$ and the same inequalities hold if we append arbitrary sequences after $x-$ and $x+$, so for instance $+-+ < + < ++-$. 
	
	\smallskip
It should be remarked that $\no$, like $\on$, is not a set but a proper class.	As an ordered class, $\no$ has a remarkable universal property: every totally ordered {\em set} can be embedded in $(\no,<)$.

	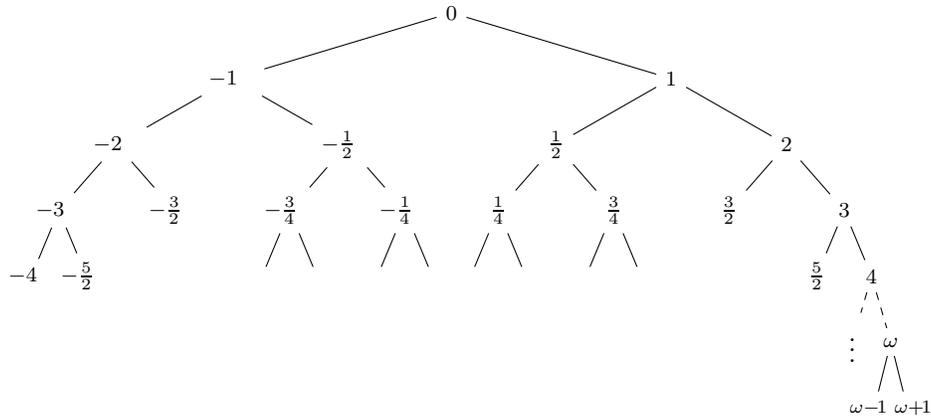
\begin{figure}[ht]
		\begin{tikzpicture}[level distance = 12mm,font=\fontsize{8pt}{10pt},scale = 0.73]
		\tikzstyle{level 1} = [sibling distance = 80mm]
		\tikzstyle{level 2} = [sibling distance = 42mm]
		\tikzstyle{level 3} = [sibling distance = 21mm]
		\tikzstyle{level 4} = [sibling distance = 10mm]
		\tikzstyle{level 5} = [sibling distance = 7mm]
		\tikzstyle{level 6} = [sibling distance = 6mm,font=\fontsize{7pt}{10pt}]
		\tikzstyle{level 7} = [sibling distance = 6mm,font=\fontsize{6pt}{8pt}]
		\node  (z){$0$}
		child {node  (-1) {$-1\phantom{-}$}
			child {node  (-2) {$-2\phantom{-}$}
				child {node {$-3\phantom{-}$}
					child {node  (-3) {$-4\phantom{-}$}}
					child {node  (e) {$-\frac{5}{2}\phantom{-}$}}
				} 
				child {node {$-\frac{3}{2}\phantom{-}$}}
			}
			child {node  (-1/2) {$-\frac{1}{2}\phantom{-}$}
				child {node {$-\frac{3}{4}\phantom{-}$}
					child {node {}}
					child {node {}}	
				}
				child {node {$-\frac{1}{4}\phantom{-}$}
					child {node {}}
					child {node {}}
				}
			}
		}
		child {node  (1) {$1$}
			child {node  (1/2) {$\frac{1}{2}$}
				child {node {$\frac{1}{4}$}
					child {node {}}
					child {node {}}
				}
				child {node {$\frac{3}{4}$}
					child {node {}}
					child {node {}}
				}
			}
			child {node  (2) {$2$}
				child {node {$\frac{3}{2}$}}
				child {node (3){$3$}
					child {node  (5/2) {$\frac{5}{2}$}}
					child {node  (4) {$4$}
						child {node (left4) {$\vdots$} edge from parent[dashed] }
						child {node (omega) {$\omega$} edge from parent[dashed] 
							child {node (omega+1) {$\omega\kern-2pt-\kern-3pt 1\;\;$} edge from parent[solid]  }
							child {node (omega-1) {$\;\;\omega\kern-2pt +\kern-3pt 1$} edge from parent[solid] }
						}
					}
				}
			}
		};
		\end{tikzpicture}
		\caption[Fig 1]{A few surreal numbers} \label{fig:tree}
	\end{figure}
	
\section{Left and right options}	
A subclass $C\subseteq \no$ is {\bf convex} if whenever $x<y<z$ are surreal numbers and $x,z$ are in $C$, also $y$ belongs to $C$. Every non-empty convex subclass of $\no$ has a simplest element, given by the element with smallest birthday. 
	If $L$ and $R$ are sets of surreals, we write $L<R$ if each element of $L$ is smaller than each element of $R$. In this case, the class of all elements $x\in \no$ satisfying $L<x<R$ is non-empty and convex. We write 
	$$x=\pair L R$$ 
	to express the fact that $x$ is the simplest surreal such that $L<x<R$. This representation is not unique, but it can be made unique adding the condition that $L\cup R$ is the set of all of all surreals simpler than $x$. This is called the {\bf canonical representation} of $x$. In this case, the elements of $L$ are called the \textbf{left-options} of $x$, while the elements of $R$ are its \textbf{right-options}. 
	\smallskip
	
		We can now name a few surreals  (see \prettyref{fig:tree}). The root of the tree is the simplest surreal, and it is written as $0 = \pair{\emptyset}{\emptyset}$. Then we have its right-successor  
		$1 = \pair{\{0\}}{\emptyset}$ and its left-successor $-1 = \pair{\emptyset}{\{0\}}$. The simplest surreal between $0$ and $1$ is $1/2 = \pair{\{0\}}{\{1\}}$, which coincides with the left-successor  of $1$. With these definitions we have $-1<0<1/2<1$. These labels are consistent with the ring operations that we shall define below.

\section{Sum and product}
We shall define the sum $x+y$ and the product $xy$ of two surreal numbers $x$ and $y$ by induction on simplicity (using the fact that simplicity is a well founded relation) so as to obtain an ordered field. Assume that we have already defined $a+y,\; x+b, \; a+b$ 
%
	%
	for all $a$ simpler than $x$ and $b$ simpler than $y$. 
	The axioms of ordered rings dictate that the operation $+$ should be strictly increasing in both arguments. This motivates the definition
	\[
	x+y=\{x^L+y,\ x+y^L\}\mid\{x^R+y,\ x+y^R\}
	\]
	where $x^L$ ranges over the left-options of $x$ and 
	and $x^R$ ranges over its right-options. 
	
	We define $-x$ exchanging recursively the left and right options: 
	$$-x = \pair{\{-x^R\}}{\{-x^L\}}.$$
	In terms of sign-expansions, $-x$ is obtained from $x$ exchanging all plus signs with minus signs. It can be verified that $(\no,<,+)$ is an ordered abelian group and $-x$ is the opposite of $x$, that is $x+(-x) = 0$. As usual we write $x-y$ for $x+(-y)$. 
	\smallskip
	
	To define the product $xy$ we assume that we have already defined $ay,\; xb, \; ab$ for all $a$ simpler than $x$ and all $b$ simpler than $y$. We now impose the distributivity law
	$(x-a)(y-b) = xy - xb-ay + ab$ and observe that this equality, together with the axioms of ordered rings, determines the sign of the difference $xy - (xb+ay-ab)$ given the relative order of $x,y,a,b$. We define $xy$ as the simplest surreal such that these signs are respected. More formally, we put
\[
xy=\{x^Ly+xy^L-x^Ly^L,\ x^Ry+xy^R-x^Ry^R\}\mid\{x^Ly+xy^R-x^Ly^R,\
x^Ry+xy^L-x^Ry^L\}.
\]
where as above $x^L$ ranges over the left-options of $x$,  $x^R$ ranges over its right-options, and similarly for $y^L$ and $y^R$. 

\medskip
\citet{Conway76} showed that with these operations $(\no,<,+,\cdot)$ is an ordered field. Moreover he showed that $\no$ is real closed: every polynomial over $\no$ which changes sign has a root in $\no$. 

\section{Embedding the reals}
 Since $\no$ is an ordered field, it contains a unique subfield $\Q\subset \no$ isomorphic to the rational numbers. The subgroup of the dyadic
	rationals $\frac{m}{2^{n}}\in\Q$, with $m\in \Z$ and $n\in\N$, correspond
exactly to the surreal numbers $s:k\to\{0,1\}$ whose birthday is a finite ordinal $k\in\N$. 

Now, given a real number $r\in \R$, let 
$L\subseteq \Q$ be the set of rationals $<r$ and let $R \subseteq \Q$ 
be the set of rationals $>r$. If we identify $r$ with the surreal $\pair{L}{R}$ we obtain inclusions of ordered fields $$\Q\subset \R \subset \no.$$ 
Under this identification, a surreal number $x:\alpha\to \{+,-\}$ belongs to $\R$ if and only its birthday is $\leq \omega$ and its sign-expansion is not eventually constant (see \cite{Conway76} or \cite[p.\ 33]{Gonshor1986}). 
By the classical work of Tarski, the theory of real closed ordered field is model complete, so the field $\R$ is an elementary substructure of $\no$. 
 
\section{Embedding the ordinals}
The ordinal numbers admit a natural embedding in the surreals. 
The ordinal $\alpha$ is mapped to the sign-expansion $x:\alpha \to \{-,+\}$ consisting entirely of plus signs. The image of this embedding is given by the surreals which admit a representation of the form $\pair L \emptyset$. 
When there is no risk of confusion we identify the ordinals with their image in $\no$ and write $\on\subset \no$. Note however that we cannot make this identification when speaking about the birthday of a surreal number.

The natural numbers $\N\subset \Q \subset \no$ coincide with the finite ordinals under the above embedding. The simplest surreal bigger than all natural numbers is the ordinal $\omega = \pair{\{0,1,2,\ldots \}}{\emptyset}$; its successor is the ordinal $\omega+1 = \pair{\{0,1,2,\ldots, \omega \}}{ \emptyset}$. 

The surreal sum and product, when restricted to $\on\subset\no$, coincide with the Hessenberg sum and product
of ordinal numbers \cite[p. 28]{Conway76}. Unlike the usual sum and product of ordinals, the Hessenberg operations are commutative. 

Since $\no$ is a real closed ordered field which includes both $\R$ and $\no$, it contains some strange numbers like $\omega-1 = \pair{\{0,1,2,\ldots \}}{\{ \omega\}}$ or $1/\omega =\pair{\{0\}}{\{ 2^{-n}\}_{n\in \N}} $, or $\sqrt{\omega}$, and we shall later see that it also contains $\log(\omega)$ and $\exp(\omega)$. 

\section{Asymptotic notations}\label{sec:asymptotic}
Given $f,g$ in an ordered abelian group, we write $f\preceq g$
if $|f|\leq n|g|$ for some $n\in\N$; if $f\preceq g$ holds, we say that $f$ is {\bf dominated} by $g$. If both $f\preceq g$ and $g\preceq f$ hold, we say that $f$ and $g$ belong to the same {\bf Archimedean class}, and we write  $f\asymp g$. We say that $f$ is {\bf strictly dominated} by $g$, written $f\prec g$, if we have both $f\preceq g$ and $f\not\asymp g$.  We define $f\sim g$ as $f-g \prec f$ and we say in this case that $f$ is {\bf asymptotic} to $g$. Notice that $\sim$ is a symmetric relation. Indeed assume $f-g \prec f$ and let us prove that $f-g\prec g$. This is clear if $f\preceq g$. On the other hand if $g \prec f$, then $f-g \asymp f$, contradicting the assumption. 

\smallskip
We shall write $O(f)$ for the set of all $g$ such that $g\preceq f$ and $o(f)$ for the set of all $g$ such that $g\prec f$. 

\smallskip
Now let $K$ be an ordered field. We can use the above notations for elements of $K$ referring to the underlying structure of additive ordered group. 
Given a multiplicative subgroup $\mathfrak{N}\subseteq K^{>0}$, we say that $\mathfrak{N}$ is a \textbf{group of monomials} of $K$ if for every $f\in K\setminus \{0\}$
there is one and only one $\n\in\mathfrak{N}$ with $f\asymp \n$. 

\smallskip
Every ordered field $K$ admits a Krull valuation $v:K^*\to v(K^*)$ whose value ring is the subring of finite elements $O(1)$. This is called the {\em natural valuation}, or {\em Archimedean valuation}. The restriction of the natural valuation to a group of monomials $\mathfrak{N}$ is an isomorphism from $\mathfrak{N}$ to $v(K^*)$. In other words, a group of monomials is a section of the natural valuation. 
\begin{exa}
	Let $\R(\x)$ be the field of rational functions ordered by $\x>\R$ and let $f,g\in \R(\x)$. We have $f\prec g$ if $f/g$ tends to $0$; 
	$f\sim g$ if $f(x)/g(x)$ tends to $1$; and $f\asymp g$ if  $f/g$ tends to a non-zero limit in $\R$. The multiplicative group $\x^\Z$ is a group of monomials of $\R(\x)$. 
\end{exa}
If $K$ is a real closed field, its value group with respect to the Archimedean valuation is a $\Q$-vector space. From the existence of basis in vector spaces it follows that every real closed field admits a group of monomials, but this property may fail if we relax the assumption that the field is real closed. 

\section{Generalised power series}
Let $(\mathfrak{N},<,\cdot,1)$ be an abelian ordered group, written in multiplicative notation. We write $\R((\mathfrak{N}))$ to denote Hahn's field of generalised power series with monomials in $\mathfrak{N}$ and we recall that $\R((\mathfrak{N}))$ is a maximal ordered field with a group of monomials isomorphic to $\mathfrak{N}$ (\citet{Kaplansky1942}). 

\begin{defn}[\citet{Hahn1907}]
An element of $\R((\mathfrak{N}))$ is a function $f:\mathfrak{N}\to \R$ whose {\bf support} $\{\n \in \mathfrak{N} \mid f(\n ) \neq 0\}$ is a reverse well ordered subset of $\mathfrak{N}$. To denote such a map, we use the notation 
$f= \sum_{\n \in \mathfrak{N}} \n f(\n )$, or the notation 
$$f=\sum_{i<\alpha}\n_{i}r_{i}$$
 where $\alpha\in \on$, $(\n_i)_{i<\alpha}$ is a decreasing enumeration of the support of $f$, and $0\neq r_i = f(\n_i)\in \R$ for all $i<\alpha$ (if $\alpha = 0$, the sum is empty, and $f=0$). 

	 The pointwise addition of functions makes $\R((\mathfrak{N}))$ into a group. 
	   The multiplication $fg$ is defined by the usual convolution formula: the coefficient of $\n$ in the product $fg$ is the sum $\sum f (\mathfrak a)g(\mathfrak b)$ taken over all pairs $(\mathfrak a, \mathfrak{b})$ with $\mathfrak{a}\mathfrak{b}= \m$. Since the supports of $f$ and $g$ are reverse well ordered, there are only finitely many pairs $(\mathfrak a, \mathfrak{b})$ such that the real number $f (\mathfrak a)g(\mathfrak b)$ is non-zero, so $fg$ is well defined. With these operations $\R((\mathfrak{N}))$ is obviously a ring, and we shall see below that it is a field. We order $\R((\mathfrak{N}))$ in the obvious way: 
	   if $f = \sum_{i<\alpha}\n_{i}r_{i} \neq 0$, then $f>0 \iff r_0 >0$. 
	   
	   \end{defn}
	   \begin{exa}
	   	The field of Laurent series in descending powers of $\x$ coincides with the Hahn field $\R((\x^\Z))$ ordered by $\x>\R$. 
	   \end{exa}	
	We introduce a notion of infinite sum in $\R((\mathfrak{N}))$ as follows. 
	\begin{defn}
			A family $(f_{i}\suchthat i\in I)$ in $\R((\mathfrak{N}))$ is {\bf summable} if
		each $\n\in \mathfrak{N}$ belongs to the support of finitely many $f_{i}$ and 
		there is no strictly increasing sequence $(\n_k)_{k \in \N}$ in $\mathfrak{N}$ such that each $\n_k$ belongs to the support of some $f_i$. 
		The sum $$f=\sum_{i\in I}f_{i}$$ is then defined adding the coefficients of the corresponding monomials. 
	\end{defn}
	Given a multi-index $i = (i_1,\ldots,i_n) \in \N^n$ and $x = (x_1,\ldots,x_n)$ in $\R((\mathfrak{N}))^n$, let $x^i = x_1^{i_1}x_2^{i_2}\cdots x_n^{i_n}$.  We write $x\prec 1$ if $x_i\prec 1$ for all $i=1,\ldots, n$.
	
\begin{lem}[Neumann's lemma  \citep{Neumann1949,Alling1987}]\label{lem:Neumann} 
	For every $\eps \prec 1$ in $\R((\mathfrak{N}))^n$ and  every $\{r_i \}_{i \in \N^n} \subseteq \R$ the family $( r_i \eps^i)_{i \in \N^n}$ is summable. 
\end{lem}
If $0\neq f \in \R((\mathfrak{N}))$, we can find the multiplicative inverse of $f$ as follows. Factoring out the leading monomial we write $f = \n r(1+\eps)$ with $\n \in \mathfrak{N}$, $r\in \R^*$ and $\eps \prec 1$. Then $f^{-1} = \n^{-1}r^{-1}(1+\eps)^{-1}$ where  $(1+\eps)^{-1}= \sum_{n\in \N} (-1)^n \eps^n$ (the sum exists by Neumann's lemma). 
	

	\section{The omega-map}
	For $x \in \no$, we say that $x$ is a {\bf surreal monomial} if $x$ is the simplest positive surreal in its Archimedean class. 
	The surreals monomials form a group under the multiplication of $\no$, so they are indeed a group of monomials of $\no$ according to the previous definitions. Since $\no$ is a proper class, its group of monomials is also a proper class. What is more interesting is that the class of monomials is order isomorphic to $\no$ and it can be parametrised as follows. 
	\begin{thm}[\citet{Conway76}]
		There is an increasing map  $$x\in \no \mapsto \oxp{x} \in \no$$ whose image $\oxp{\no}$ is the class of surreal monomials and such that
		\begin{enumerate}
			\item $\oxp{0} = 1$;
			\item $\oxp{1} = \omega$ (the ordinal $\omega$ seen as a surreal); 
			\item $\oxp{x+y} = \oxp{x} \oxp{y}$. 
		\end{enumerate}
	This implies in particular that $\omega^{(xn)} = (\omega^x)^n$ for each $n\in \Z$, so we can write $\omega^{xn}$ without parenthesis. 
	\end{thm}
	The definition of $\oxp{x}$ is the following.  
	\begin{defn}\label{defn:omega-map}
		For $x\in \no$, let $\oxp{x}=\{0,k\oxp{x'} \} \mid \{2^{-k} \oxp{x''} \}$ where $x'$ ranges over the left-options of $x$, $x''$ ranges over its right-options, and $k$ ranges in $\N$. 
	\end{defn}
	
	It follows from the definition that $\oxp{x}$ is the simplest representative $>0$ of its Archimedean class and that if $x$ is simpler than $y$, then $\oxp{x}$ is simpler than $\oxp{y}$. 
	
	\section{Conway normal form}\label{sec:normal}
	Let $\M= \oxp{\no}$ be the class of surreal monomials and let  $\R((\oxp{\no}))_{\on}=\R((\M))_\on$ be the field of generalised power series with monomials in $\M$; the subscript ``$\on$'' is meant to emphasise that, although the group $\oxp{\no}$ is a proper class, in the definition of  $\R((\oxp{\no}))_{\on}$ we only consider series $\sum_{i<\alpha} \m_i r_i$ whose support is a set (indexed by an ordinal). In this section we shall see that 
	$\no=\R((\omega^\no))_\on$ via a canonical identification.
	
		\begin{rem} Let $\kappa$ be a regular uncountable cardinal and consider the subfield $\no(\kappa) \subset \no$ whose elements are the surreal numbers born before day $\kappa$. It can be shown that $\no(\kappa)$ is isomorphic to a field of the form $\R((\mathfrak{N}))_\kappa$, where $\R((\mathfrak{N}))_\kappa \subseteq \R((\mathfrak{N}))$ consists of the series with support of cardinality $<\kappa$ \cite{DriesE2001}. So $\no$ should not be thought as a field of the form $\R((\mathfrak{N}))$, but rather as a field of the form $\R((\mathfrak{N}))_\kappa$ with $\kappa$ being an inaccessible cardinal living in some larger set theoretic universe. 
		\end{rem}  
	
	In \S \ref{sec:exp} we shall see that $\no$ admits an exponential map. 
	This should be contrasted with a result of \citet*{Kuhlmann1997} where it is shown that a field of the form $\R((\mathfrak{N}))$ (where $\mathfrak{N}$ is a {\em set} of monomials) never admits an exponential map. 

		\begin{thm}[\citet{Conway76}]\label{thm:normal-form} We can make the identification
		$\no = \R((\oxp{\no}))_\on$ via a canonical isomorphism of ordered fields. 
	\end{thm}
	\begin{proof}
		By induction on the ordinal $\alpha$, we will associate to a generalised power series $\sum_{i<\alpha} \m_i r_i\in \R((\oxp{\no}))_\on$ a surreal number $f\in \no$ and we call $\sum_{i<\alpha} \m_i r_i$ the {\bf normal form} of $f$. In this case we write $$f=\sum_{i<\alpha} \m_i r_i,$$ identifying the surreal $f$ with its normal form. The strategy is to first define a field embedding from $\R((\omega^\no))_\on$ to $\no$, and then show that it is surjective. Clearly $0$ goes to $0$ under the embedding. If $\alpha=\beta+1$, we put 
		$$\sum_{i<\alpha} \m_i r_i = \sum_{i<\beta}\m_i r_i + \m_\beta r_\beta$$
		  where ``$+$'' is the addition in $\no$; if $\alpha$ is a limit ordinal, then the image of $\sum_{i<\alpha} \m_i r_i$ in $\no$ is the simplest $z\in \no$ such that $z-\sum_{i<\beta} \m_i r_i \asymp \m_\beta$ for all $\beta<\alpha$ (we can consider $\sum_{i<\beta} \m_i r_i$ as an element of $\no$ by the inductive hypothesis).  The existence of $z$ follows from the fact that every convex class of surreal numbers has a simplest element. 
		  
We have thus defined a field embedding from $\R((\oxp{\no}))_\on$ to $\no$. To prove that it is surjective we define the inverse map, that is, we compute the normal form of a surreal number. So let $f\in \no$. If $f=0$, then $f$ is already in normal form (represented by the empty sum). If $f\neq 0$ there is a unique monomial $\m$ and a unique real number $r\in \R^*$ such that $f=r\m + g$ with $g\prec f$. Then $r\m$ is the first term of the normal form of $f$ and to find the other terms we iterate the process. More precisely, suppose we have defined the $i$-th term $r_i \m_i$ of $f$ for each $i<\alpha$, so that we can write $$f= \sum_{i<\alpha} r_i \m_i + g_\alpha$$ with $g_\alpha\prec \m_i$ for all $i<\alpha$. If $g_\alpha=0$ we have finished. In the opposite case, the $\alpha$-th term of $f$ is first term of $g_{\alpha}$.  It can be shown that $$\alpha \leq \birth(\sum_{i<\alpha} r_i \m_i) \leq \birth(f),$$ so the process must stop in a number of steps $\leq \birth(f)$. 
\end{proof}
%
\begin{defn}\label{defn:Conway-nf}
Let  
$f=\sum_{i<\alpha}\m_i r_i \in \no$
be written normal form, as in the proof of \prettyref{thm:normal-form}. 
Since the omega-map parametrises the surreal monomials, we can also write
 	\begin{equation*}\label{eq:Conway-nf}
 	f = \sum_{i<\alpha} \oxp{x_i}r_i
 	\end{equation*}
 	where $(x_i)_{i<\alpha}$  is a decreasing sequence in $\no$.  This is called the {\bf Conway normal form} of $f$. 

 Thanks to the identification $\no=\R((\M))_\on$, we can define the {\bf support} of a surreal number and the sum of a summable family of surreal numbers as in the context of generalised series. Note that the support of $0$ is the empty set.
\end{defn}

\begin{prop}[\cite{Conway76}] Conway's omega-map extends the homonymous map on the ordinal numbers. 
	If $f\in \on \subset \no$ is an ordinal, its Conway normal form $\sum_{i<\alpha}\omega^{x_i} r_i$ coincides with the Cantor normal form. So in this case $x_i\in \on$, $r_i\in \N$, and $\alpha <\omega$.     		
\end{prop}

The class $O(1)$ of the finite elements of $\no$ is an $\R$-vector subspace of $\no$, and as such it has many complementary spaces. We can however select one specific complement as follows. 
\begin{defn}
	\label{defn:purely-infinite}
Let $\M = \oxp{\no}$ be the class of all surreal monomials and notice that $\M^{>1} = \oxp{\no^{>0}}$ is the class of all infinite monomials. 
A surreal $x = \sum_{i<\alpha} \m_i r_i \in \R((\M))_{\on}$ is  {\bf purely infinite} if all monomials $\m_i$ in its support are $>1$ (hence infinite). Let $\no^\uparrow$ be the (non-unitary) ring of purely infinite surreals. Every $x\in \no$ can be written in a unique way in the form 
$$x = x^\uparrow + x^\circ + x^\downarrow$$
where $x^\uparrow \in \no^\uparrow$, $x^\circ \in \R$ and $x^\downarrow \prec 1$. 
This yields a direct sum decomposition $$\no = \no^\uparrow +\R + o(1)$$ of $\R$-vector spaces, where $o(1)$ is the set of elements $\prec 1$. Note that $\no^\uparrow$ is a complement of the ring of {\bf finite} elements $O(1)=\R+o(1) = \{x\in \no \mid x \preceq 1\}$. 
\end{defn}

\section{Restricted analytic functions}
In this section we show that every real analytic function $f:\R\to \R$ has a natural extension to a function $f:O(1)\to \no$ where $O(1)$ is the class of finite surreal numbers. More generally, a real analytic functions defined on an open subset $U\subseteq \R^n$, has an extension to a function $f:U+o(1)\to \no$ as in the following definition. 

\begin{defn}\label{defn:extension}
	Let $U\subseteq \R^n$ be an open set and let $f:U\to \R$ be a real analytic function. 
	Now let 
	$$\widetilde{U} = U+o(1)$$
	be the infinitesimal neighbourhood of $U$ in $\no^n$. There is a natural extension of $f$ to a function 
	$$\tilde{f}:\widetilde{U} \to \no$$
	defined as follows. For $r\in U$, let $\sum_{i\in \N^n} \frac{D^i f(r)}{i!} X^i$ be the Taylor series of $f$ around $r$, where $i=(i_1,\ldots,i_n)$ is a multi-index. Now for $\eps\in o(1)^n \subseteq \no^n$, define $\tilde{f}(r+\eps) = \sum_{i\in \N^n}\frac{D^i f(r)}{i!} \eps^i$, where the summability is ensured by \prettyref{lem:Neumann} (Neumann's lemma). Since $f$ is analytic on $U$, the function $\tilde{f}:\widetilde{U}\to \no$ extends $f$. 
\end{defn}	

\begin{rem}\label{rem:composition} We have
	$\widetilde{f\circ g}= \widetilde{f}\circ \widetilde{g}$ whenever the image of $g$ is contained in the domain of $f$.
\end{rem}

\begin{defn} Let $\R$ be the field of real numbers, let $\R_{an}$ be the expansion of $\R$ with all analytic functions restricted to some box $[-1,1]^n$, and let $\R_{an}(\exp)$ be the expansion of $\R_{an}$ with the real exponential function. Now let $T_{an}$ be the theory of $\R_{an}$, let $T_{\exp}$ be the theory of $(\R,\exp)$ and $T_{an}(\exp)$ be the theory of $\R_{an}(\exp)$.
\end{defn}
We recall that $T_{\exp}$, $T_{an}$ and $T_{an}(\exp)$ are model complete by \cite{Wilkie1996,Denef1988,Dries1994} respectively.

\begin{thm}[\cite{DriesE2001}] \label{thm:Tan} For each analytic function $f:U\to \R$ whose domain includes $[-1,1]^n$ consider the function $\tilde{f}:\widetilde{U}\to \no$ in 
	\prettyref{defn:extension} and its restriction to $[-1,1]^n$. The expansion of  $\no$ with all these restricted functions is a model of $T_{an}$. 
\end{thm}
The proof is based on the axiomatisation of $T_{an}$ given in \cite{Dries1994}.

\section{Exponentiation}\label{sec:exp}
The next goal is to expand $\no$ to a model of $T_{an}(\exp)$
through the introduction of an exponential function. We already know how to extend the real exponential function $\exp:\R\to \R$ to a function $\exp:O(1)\to \no$ using \prettyref{defn:extension}, so the problem is how to define $\exp(x)$ when $x\in \no$ is infinite. 
Since we want the surreal $\exp$ to be an isomorphism of ordered groups from $(\no,+,<)$ to $(\no^{>0},\cdot, <)$, the image $\exp(\no^\uparrow)$ of the class of purely infinite elements (\prettyref{defn:purely-infinite}) must be a complement of $O(1)^{>0}$ in the multiplicative group $\no^{>0}$. We already have a natural choice for such a complement, namely the class $\M= \omega^\no$ of surreal monomial. 
It is then natural to require that 
$$\exp(\no^\uparrow) = \M,$$ 
so in particular $\exp(\M^{>1}) \subset \M$, i.e. the class of infinite monomials is closed under $\exp$. 
To achieve this goal, for $x>0$ we define $\exp(\omega^x) = \omega^{\omega^{g(x)}}$ for a suitable increasing bijection $g:\no^{>0}\to \no$. To ensure that $\exp$ grows faster than any polynomial, the function $g$ is chosen in such a way that for $\m\in \M^{>1}$, we have $\exp(\m ) > \m^n$ for all $n\in \N$. This translates into the condition  $\omega^{g(x)} \succ x$, so we define $g$ as the simplest increasing map with this property. The following definition formalises the idea.   

\def\Gonshor169{\citet[p. 169]{Gonshor1986}}
\begin{defn}[\Gonshor169]\label{defn:g-h}
	The {\bf index} of $x\in \no^{\neq 0}$ is the unique $c= \ind(x)$ such that $x\asymp \omega^c$. In particular $\ind(\omega^c) = c$. Define $g:\no^{>0}\to \no$ by the recursive equation 
	\begin{equation}
	g(x) =
	\{\ind(x), g(x')\} \mid  \{g(x'')\}	 
	\end{equation}
	where $x'$ ranges over the left options of $x$ and $x''$ ranges over its right options. 
\end{defn} 
\begin{rem}
	 It can be proved that $g$ is an increasing bijection $g: \no^{>0}\to \no$ (with inverse given by \prettyref{defn:h}). Since $g(x) > \ind (x)$, we have $\omega^{g(x)} \succ x$.
\end{rem}
The function $g$ can be difficult to compute, but \citet{Gonshor1986} showed that $g(n) = n$ for every $n\in \N$. More generally, if $\alpha$ is an ordinal, then $g(\alpha) = \alpha$ unless there is an epsilon number $\eps$ such that $\eps \leq \alpha < \eps + \omega$, in which case $g(\alpha) = \alpha+1$ \cite[Thm. 10.14]{Gonshor1986}. 
\begin{defn}\label{defn:exp}
	Given $x\in \no$, we write 
	$x = x^\uparrow + x^\circ + x^\downarrow$ as in \prettyref{defn:purely-infinite} and define 
	$\exp(x) = \exp(x^\uparrow)\exp(x^\circ) \exp(x^\downarrow)$
	where the factors on the right-hand side are defined as follows: 
	\begin{enumerate}
		\item The restriction of $\exp$ to $\R$ coincides with the real exponential function. 
		\item $\exp(\eps) = \sum_{n\in \N} \frac{\eps^n}{n!}$ for $\eps\prec 1$.
		\item $\exp(\omega^x) = \omega^{\omega^{g(x)}}$ for  $x>0$. This defines the restriction of $\exp$ to the class $\M^{>1}$ of all infinite monomials. 
		\item The extension to $\no^\uparrow$ is given by the formula $$\exp(\sum_{i<\alpha} \omega^{x_i} r_i) = \omega^{\sum_{i<\alpha} \omega^{g(x_i)} r_i}$$ where $x = \sum_{i<\alpha} \omega^{x_i} r_i\in \no^\uparrow$ (i.e. $x_i>0$ for all $i$). 
	\end{enumerate}
\end{defn}
  The above definition is equivalent to the one of \citeauthor{Gonshor1986} \cite[Thms. 10.2, 10.3, 10.13]{Gonshor1986}. 
  Note that  (1) and (2) agree with \prettyref{defn:extension}. 
  \begin{thm}
  	\label{thm:Gonshor} We have: 
  	\begin{enumerate}
  		\item $\exp$ is a group isomorphism from $(\no,+,<)$ to $(\no^{>0},\cdot,<)$; 
  		\item $\exp$ extends the real exponential function; 
  		\item $\exp(\eps) = \sum\limits_{n=0}^\infty \frac{\eps^n}{n!}$ for $\eps\prec 1$; 
  		\item   $\exp(\no^\uparrow)$ is the class $\M$ of surreal monomials;
  		\item $\exp(x) > x^n$ for all $x\in \no^{>\N}$ and all $n\in \N$. 
  	\end{enumerate}
  \end{thm}	
  \begin{proof}
  	All points follow easily from the definitions except possibly (5). By \prettyref{defn:g-h} we have $g(z) > \ind(z)$ for all $z\in \no^{>0}$. We thus obtain $\omega^{g(z)} \succ \omega^{\ind(z)} \asymp z$. It follows that for $z>0$ we have $\exp(\omega^z) = \omega^{\omega^{g(z)}} > \omega^{zn}$ for all $n\in \N$, so (5) holds whenever $x = \omega^z$ is an infinite monomial and the general case easily follows.  
  \end{proof}
  
\section{Logarithm}	
The surreal logarithm $\log:\no^{>0}\to \no$ is the inverse of $\exp:\no\to \no^{>0}$. To be able to give a direct definition we need an auxiliary function. 
\begin{defn}\label{defn:h}
	Define 
	$h:\no\to \no^{>0}$  by 
	\begin{equation}
	h(x) =
	\{0, h(x')\} \mid  \{h(x''), \oxp{x}/2^n\}	 
	\end{equation}
	where $x'$ ranges over the left-options of $x$, $x''$ ranges over its right-options, while $n$ ranges over the natural numbers. 	
\end{defn}

\begin{rem}\label{rem:h}
	\citet{Gonshor1986} showed that 
	$h$ is the inverse of the function $g:\no^{>0}\to \no$ in \prettyref{defn:g-h}.
\end{rem}
	
	We are now ready to define $\log:\no^{>0}\to \no$. 
	\begin{defn}\label{defn:log}
	Given $x\in \no^{>0}$, we write $x = r\m(1+\eps)$, with $r\in \R^{>0}$, $\m\in \M$ and $\eps \prec 1$, and we define $\log(x) = \log(r) + \log(\m) + \log(1+\eps)$, where the right-hand side is defined as follows. 
	\begin{enumerate}
		\item The restriction of $\log$ to $\R^{>0}$ agrees with the natural logarithm on $\R$;
	\item 	$\log(1+\eps) = \sum_{n=1}^\infty \frac{(-1)^{n+1}}{n} \eps^n$ for $\eps\prec 1$;
	\item $\log(\omega^{\omega^x}) = \omega^{h(x)}$;
		\item $\log(\oxp{\sum_{i<\alpha} \oxp{x_i} r_i})=\sum_{i<\alpha}\oxp{h(x_i)}r_i$. 
	\end{enumerate}

\end{defn}	
	
	The following proposition is an easy consequence of Remarks \ref{rem:composition} and \ref{rem:h}.
	\begin{prop}
		$\log:\no^{>0}\to \no$ is the inverse of $\exp:\no\to \no^{>0}$.  
	\end{prop}
		In particular, since $\exp(\no^\uparrow) = \M$, we have $\log(\M) = \no^\uparrow$.
	
\section{Ressayre's axioms}
To prove that $\no$ is an elementary extension of $\R_{\exp}$ we need the following result of \citet{Ressayre1993}, which can also be derived from the axiomatisation of $T_{an}(\exp)$ given by \citet*{DriesMM2001}. 
 \begin{thm}[\cite{Ressayre1993,DriesMM2001}] \label{thm:Ressayre}A real closed ordered field $K$ endowed with an  isomorphism of ordered groups $E: (K,+,<)\to (K^{>0},\cdot,<)$ is a model of $T_{\exp}$ if and only if the following axioms hold: 
 \begin{enumerate}
 	\item[(i)] $E(x) >x^n$ for all $x\in K^{>\N}$ and all $n\in \N$ ;
\item[(ii)]  the restriction of $E$ to $[0,1]$ makes $K$ into a model of the theory of $(\R,\exp\rest [0,1])$. 
\end{enumerate}
\end{thm}
\begin{rem}
If $(K, E\rest [0,1])$ is elementary equivalent to $(\R,\exp \rest [0,1])$, then these two structures have isomorphic ultrapowers $K^*$ and $\R^*$. Since the ultrapower construction yields elementary extensions in any language, we can use the isomorphism to expand $K^*$ to a model of $T_{an}$.  It follows that any structure $(K,E)$ satisfying the hypothesis of \prettyref{thm:Ressayre} (Ressayre's axioms) has an elementary extension which can be expanded to a model of $T_{an}(\exp)$. 
\end{rem}

\begin{cor}[\citet{DriesE2001}]\label{cor:Dries2001}
	The field $\no$ of surreal numbers, with Gonshor's $\exp$ and the natural interpretation of the restricted analytic function, is a model of $T_{an}({\exp})$. 
\end{cor}
Since $T_{an}(\exp)$ is model complete, it follows that $\no$ is an elementary extension of $\R$ as an exponential field with all restricted analytic functions.

		\begin{rem} In the presence of the other hypothesis, 
			condition (i) in \prettyref{thm:Ressayre} is equivalent to: 
			$E(x) \geq x+1$ for all $x\in K$. The latter is a first order axiom, so the first-order theory $T_{\exp}$ is finitely axiomatisable over the theory of restricted $\exp$. As a consequence, these two theories are either both decidable or both undecidable. \citet{Macintyre1996a} prove the decidability of $T_{\exp}$ assuming ``Schanuel's conjecture'', but the unconditional decidability remains an open problem.  
		\end{rem}
		
\section{Exponential normal form}
The presence of $\exp$ and the omega-map generates a conflict of notation, as $\omega^x$ could be interpreted either as the omega-map applied to $x$ or as $\exp(x \log(\omega))$. To avoid ambiguities, when there is a risk of confusion we write $\orexp{x}$ (with a little dot) for the omega-map. 
\begin{defn}
Given $a,b\in \no$ with $a>0$, we define 
$$a^b = \exp(b \log(a))$$ 
and we write $\orexp{b}$ or $\Omega(b)$ for the image of $b$ under the omega-map. 
\end{defn}
In general $\Omega(b) = \orexp{b}\neq \omega^b = \exp(b \log(\omega))$. Indeed $\orexp{b}$ is always a monomial, while $\omega^b$ can be any positive surreal. 

From $\no=\R((\M))_\on$ (\prettyref{thm:normal-form}) and $\M = \exp(\no^\uparrow) = \orexp{\no}$ (\prettyref{thm:Gonshor}), we obtain the following result.   
\begin{prop}
	Every surreal number $f$ can be written in a unique way in the form 
	\begin{equation*}
	f = \sum_{i<\alpha} e^{\gamma_i} r_i
	\end{equation*}
	where 
	$\alpha$ is an ordinal, $(\gamma_i)_{i<\alpha}$ is a decreasing sequence in $\no^\uparrow$ and $r_i \in \R^*$.  
This is called the {\bf Ressayre form} of $f$, or the {\bf exponential normal form}.
\end{prop}
To convert the Conway normal form to the Ressayre form we need to show how to pass from the representation $\M = \orexp{\no}$ to the representation $\M= \exp(\no^\uparrow)$. It is convenient to introduce the following definition.  


\begin{defn} \label{defn:H-G} Let $g: \no^{>0}\to \no$ and $h: \no\to \no^{>0}$ be as in Definitions \ref{defn:g-h} and \ref{defn:h} and recall that $h$ is the inverse of $g$. Define $G: \no^\uparrow \to \no$ and $H:\no \to \no^\uparrow$ as follows. 
	\begin{enumerate}
			\item $G(\sum_{i<\alpha} \orexp{y_i} r_i) = \sum_{i<\alpha} \orexp{g(y_i)} r_i$; 
			\item $H(\sum_{i<\alpha} \orexp{x_i} r_i) = \sum_{i<\alpha} \orexp{h(x_i)} r_i$.
	\end{enumerate}
Clearly $H$ and $G$ are strongly $\R$-linear (i.e. they are $\R$-linear and distributive over infinite sums) and $H$ is the inverse of $G$. 
\end{defn}
The following proposition is a rephrasing of Definitions \ref{defn:exp}(4) and \ref{defn:log}(4). 
\begin{prop} We have: 
	\begin{enumerate}
		\item 
	$e^\gamma = \orexp{G(\gamma)}$ for all $\gamma \in \no^\uparrow$; 
		\item 
		$\orexp{x} = e^{H(x)}$ for all $x\in \no$. 
		\end{enumerate}
\end{prop}	

\begin{cor}
 For $r\in \R$ and $x\in \no$, $\Omega(x)^r= \Omega(xr)$, so a real power of a monomial is a monomial. In particular, $\omega^r = \orexp{r}$. 
\end{cor}
\begin{proof} By
	the $\R$-linearity of $H$, $\Omega(x)^r = (e^{H(x)})^r  = e^{H(x)r} = e^{H(xr)} = \Omega(xr)$  
\end{proof}
	
%
%
\section{Normal forms as asymptotic expansions}
Given a function $f(x)$ defined as a composition of algebraic operations, $\log$ and $\exp$, we may consider $f(\omega)$ as a surreal number, and the normal form of $f(\omega)$ corresponds to an asymptotic development of $f(x)$ for $x\to +\infty$. 
\begin{exa}
To find the normal form of $(\omega+1)^\omega\in \no$ we write 
\begin{align*}
(\omega+1)^\omega &= \exp\left(\omega(\log(1+\omega))\right)\\
& = \exp\left(\omega \left(\log(\omega) + \log(1+\omega^{-1})\right)\right)\\
&= \exp\left(\omega  \log(\omega) + \sum_{n=1}^{\infty} \frac{(-1)^{n+1}}{n}\omega^{-n+1} \right) \\
& = \exp\left(\omega  \log(\omega) + 1 - \frac{1}{2}\omega^{-1} + \ldots \right)\\
&= \omega^\omega e^1 \exp(-2^{-1}\omega^{-1} + \ldots) \\
& = e \omega^\omega \left( 1 -2^{-1}\omega^{-1} + \ldots\right)\\
& = e \omega^\omega - e 2^{-1}\omega^{\omega -1} + \ldots
\end{align*}
Replacing $\omega$ with a real variable $\x$ this corresponds to the asymptotic development 
$$(\x+1)^\x = e \x^\x - e2^{-1} \x^{\x-1} + \ldots$$
for $\x\to +\infty$. 
\end{exa}

The idea of normal forms as asymptotic expansions in exploited in \cite{Berarducci2019a} to study the possible limits at $+\infty$ of the quotients of two ``Skolem functions'', improving some results of \cite{Dries1984} on a problem of Skolem \citep{Skolem1956}. 

\section{Infinite products}
In $\no$ we have a notion of infinite sum due to the identification $\no = \R((\M))_\on$, but we also have an $\exp$ and a $\log$ function, so we can introduce infinite products via the formula 
\begin{equation*}
\prod_{i\in I} f_i = \exp\left(\sum_{i\in I}\log(f_i)\right)
\end{equation*}
where $(f_i)_{i\in I}$ in $\no^{>0}$ is such that $(\log(f_i))_{i\in I}$ is summable. We say in this case that the product of the family $(f_i)_{i\in I}$ exists. 
%
It follows at once from the definition that if $(x_i)_{i\in I}$ is summable, then the product of the family $(e^{x_i})_{i\in I}$ exists and 
$$\prod_{i\in I} e^{x_i} = e^{\sum_{i\in I} x_i}.$$
As observed in \cite{Aschenbrenner2019}, a similar statement holds for the omega-map: if $(x_i)_{i\in I}$ is summable, then the product of the family $(\orexp{x_i})_{i\in I}$ exists and 
\begin{equation*}\label{eq:prod-monomials}
\prod_{i\in I} \orexp{x_i} = \orexp{\sum_{i\in I} x_i}.
\end{equation*}
For the proof it suffices to observe that $\orexp{x_i} = e^{H(x_i)}$ where $H:\no\to \no^\uparrow$ is the strongly $\R$-linear isomorphism in \prettyref{defn:H-G}. In particular we have: 
\begin{prop} \label{prop:prodM} Given a family of monomials $(\m_i)_{i\in I}$, if the product $\prod_{i\in I}\m_i$ exists, it is a monomial.
\end{prop} 
The notation of infinite products will be employed in the definition of a surreal derivation. In general the infinite distributivity law fails, namely $\prod_{i\in I}\sum_{j\in J}x_{i,j} \neq \sum_{f:I\to J} \prod_{i\in I} x_{i,f(j)}$. Indeed it is easy to find examples in which the left-hand side is well defined, but the right-hand side is not because of summability problems. 



\section{Levels and log-atomic numbers}
\begin{defn}
	Given $x,y \in \no^{>\N}$, we say that $x$ and $y$ have the same {\bf level} if $\log_n(x) \asymp \log_n(y)$ for some $n\in \N$. If $x$ and $y$ have different levels and $x<y$, we say that $x$ has lower level than $y$,  or that $y$ has higher level than $x$. 
\end{defn}
By definition, $x$ has lower level than $y$ if and only if $\log_n(x) \prec \log_n(y)$ for every $n\in \N$, or equivalently $\exp_n(k\log_n(x))<y$ for every $n,k\in \N$. 

To be in the same level is a coarser equivalence relation than to be in the same Archimedean class: given $x\in \no^{>\N}$, the elements $x$ and $x^n$ are in the same level for each $n\in \N$, while $\exp(x)$ is in a higher level. The class of levels has the following density property.  
\begin{prop}
	Given two subsets $A$ and $B$ of $\no^{>\N}$ such that the level of every element of $A$ is smaller than the level of every element of $B$, there is $z \in \no$ which has higher level than all elements of $A$ and lower level than all elements of $B$. 	
\end{prop}
\begin{proof}It suffices to take $z = \pair{A'}{B'}$ where $A'\supseteq A$ and $B'\supseteq B$ are defined by $A'=\{ \exp_n(k\log_n(a)) \mid n,k\in \N, a\in A\}$ and $B' = \{\exp_n(2^{-k} \log_n(b)) \mid n,k\in \N, b\in B\}$. 
\end{proof}
It follows from the above density properties that there are intermediate levels between $\omega$ and $\exp(\omega)$. Our next goal is to parametrize a class of representative of the levels by $\no$ itself.    
\begin{defn}\label{defn:lambda}
	By induction on simplicity, we define $\lambda_x\in \no^{>\N}$ as the simplest element such that:
	\begin{enumerate}
		\item  $\lambda_x$ is the simplest element of its level,
		\item $\lambda_{x'} < \lambda_x < \lambda_{x''}$ for all left-options $x'$ of $x$ and all right-options $x''$ of $x$.
	\end{enumerate} 
\end{defn}
The above is equivalent to \cite[Definition 5.12]{Berarducci2018}. 
\begin{prop}
	The map $x\mapsto \lambda_x$ is increasing and its image $\lambda_\no = \{\lambda_x \mid x\in \no\}$ contains one and only one representative for each level.
\end{prop}

\begin{rem} We have: 
	\begin{enumerate}
		\item The elements $\lambda_{-\alpha}$ with $\alpha \in \on$ are coinitial in $\no^{>\N}$. 
		\item The elements $\lambda_\alpha$ with $\alpha\in \on$ are cofinal in $\no$. 
	\end{enumerate}
\end{rem}


\begin{defn}
	For $n\in \N$ we define: 
	\begin{enumerate}
		\item $\log_0(x) = x$,  $\log_{n+1}(x) = \log(\log_n(x))$; 
		\item $\exp_0(x) = x$,  $\exp_{n+1}(x) = \exp(\exp_n(x))$.
	\end{enumerate}  
	We extend the definition to the case when $n\in \Z$ putting $\log_n= \exp_{-n}$. 
\end{defn}
\begin{prop}[\cite{Berarducci2018}] \label{prop:lambda-n} We have $\lambda_{-n} = \log_n(\omega)$ for all $n\in \Z$. In particular, $\lambda_{-1} = \log(\omega), \lambda_0 = \omega, \lambda_1 = \exp(\omega)$. 
\end{prop}
\begin{defn} \label{defn:log-atomic}
	Given $x\in \no^{>\N}$, we say that $x$ is {\bf log-atomic} if, for each $n\in \N$, $\log_n(x)$ is a monomial. 
\end{defn}
\begin{thm}[\cite{Berarducci2018}]\label{thm:log-atomic}
	For $z\in \no^{>\N}$ the following are equivalent: 
	\begin{enumerate}
		\item $z$ is the simplest element of its level;
		\item $z$ is log-atomic; 
		\item $z \in \lambda_\no$.  
	\end{enumerate}
\end{thm} 
Since $x\mapsto \lambda_x$ is increasing, it follows that $\lambda_{1/2}$ represents an intermediate level between $\omega= \lambda_0$ and $\exp(\omega)= \lambda_1$, while $\lambda_{-\omega}$ is lower than $\lambda_{-n} = \log_n(\omega)$ for all $n\in \N$. 
\prettyref{prop:lambda-n} can be extended as follows. 
\begin{thm}[\cite{Aschenbrenner2019}] We have $\lambda_{x-1} = \log(\lambda_x)$ for all $x\in \no$.	
\end{thm}

When $\alpha$ is an ordinal it is suggestive to think of $\lambda_{-\alpha}$ as the $\alpha$-times iterated $\log$ applied to $\omega$ and write $\lambda_{-\alpha} = \log_\alpha(\omega)$ as in \cite{Aschenbrenner2019}. 
\begin{rem}[\cite{Aschenbrenner2019}] We have:
	\begin{enumerate}
		\item 	For each ordinal $\alpha$, $\lambda_{-\alpha} = \omega^{\omega^{-\alpha}}$; 
		\item For each limit ordinal $\alpha$, $\lambda_{-\alpha}$ is the simplest surreal $x>\N$ such that $x < \lambda_{-\beta}$ for all ordinals $\beta<\alpha$. 
	\end{enumerate}
\end{rem}
The search for a satisfactory definition of $\log_\alpha(x)$ for a general $x\in \no^{>\N}$  is connected to the problem of representing $\no$ as a field of ``hyperseries''. 
For some recent work on this project see \cite{Aschenbrenner2017g,VandenDries2018,Bagayoko2019}.

\section{Simplicity and the omega-map} \label{sec:chopping}
We have encountered two important representations of surreal numbers: sign-expansions and Conway's normal forms. In general it is not easy to convert one to the other, but the following theorem (and related results in \cite{Ehrlich2011}) shows that normal forms are well behaved with respect to the simplicity relation. 

\begin{defn}
	Given $f,g\in \no$, we say that $g$ is a {\bf truncation} of $f$, if $f=\sum_{i<\alpha} \m_i r_i$ and $g = \sum_{i<\beta} \m_i r_i$ for some $\beta<\alpha$, where both expressions are in normal form. 
\end{defn}

\begin{thm}[\protect{\citet{Conway76}, \citet[\S 6]{Gonshor1986}}]\label{thm:omega-simplicity} We have:
	\begin{enumerate}
		\item If $x$ is simpler than $y$, then $\orexp{x}$ is simpler than $\orexp{y}$.
		\item If $\m$ is a monomial and $r\in \R^*$, then $\pm \m$ is simpler or equal to $\m r$, where $\pm 1$ is the sign of $r$. 
		\item $\birth(x) \leq \birth(\orexp{x})$. 
		\item If $g$ is a proper truncation of $f$, then $g$ is simpler than $f$. 
		\item if $\m_i r_i$ is a term of $f = \sum_{i<\alpha} \m_ir_i$, then $\birth(\m_ir_i) \leq \birth(f)$; if $i+1<\alpha$  the inequality is strict because in this case $\m_ir_i$ is a term of a proper truncation of $f$. 
	\end{enumerate}
\end{thm}

\begin{cor} \label{cor:simplicity-index}If $f=\sum_{i<\alpha}\orexp{x_i}r_i$ is in normal form, then for every $i<\alpha$ we have $\birth(x_i) \leq \birth(f)$ and the inequality is strict if $i+1<\alpha$. 
\end{cor}	
If $f$ is as in the corollary, the only possibility to have $\birth(x_i) = \birth(f)$, is that $\alpha$ is a successor ordinal and $\orexp{x_i}$ is the smallest monomial in the support of $f$, that is, $i+1=\alpha$. Moreover in this case $r_{i+1} = \pm 1$, so we can write $f = \sum_{j<i} \orexp{x_i} r_j \pm \omega^{x_i}$. 
A special case of this situation is when $f = \orexp{f}$. The surreal numbers with this property are called {\bf epsilon-numbers}, and they include the ordinal numbers with the homonymous property. 

\section{Simplicity and exp}
In general it is not true that if $x$ is simpler than $y$, then $\exp(x)$ is simpler than $\exp(y)$: it suffices to take $x = \omega$ and $y= \log(\omega)$. It can however be proved that if $x$ is a truncation of $y$, then $\exp(x)$ is simpler then $\exp(y)$. This suggests the following definition, which can be useful in inductive proofs based on the exponential normal form. 

\begin{defn}
	Define $\ntrunceq$ as the smallest preorder on $\no^*$ such that 
	\begin{enumerate}
		\item if $g$ is a non-zero truncation of $f$, then $g\ntrunceq f$; 
		\item if $g$ and $f$ are purely infinite and $g \ntrunceq f$, then $\pm e^g \ntrunceq e^f r$, where $r\in \R^*$ and $\pm$ is the sign of $r$. 
		More generally, $x \pm e^g \ntrunceq x + e^f r$ where $x\in \no$ is such that all the monomials in the support of $x$ are greater than both $e^g$ and $e^f$.  
	\end{enumerate} 
	If $f \ntrunceq g$, we say that $f$ is a {\bf nested truncation} of $g$.
	We write $f \ntrunc g$ if we have both $f \ntrunceq g$ and $f\neq g$. 
\end{defn}
\begin{thm}[\protect{\cite[Thm. 4.26]{Berarducci2018}}] 
	If $y  \ntrunc x$, then $y$ is simpler than $x$. 
\end{thm}
It follows that $\ntrunc$ is a well founded relation, so we have an associated notion of rank defined by transfinite induction. 
\begin{defn}
	The nested truncation rank $\nr(x)\in \on$ of $x\in \no^*$ is recursively defined by $\nr(x) = \sup \{\nr(y)+1 \mid y \ntrunc x\}$.  
\end{defn} 

We have now an analogue of \prettyref{thm:omega-simplicity} and \prettyref{cor:simplicity-index}. 
\begin{prop}[\cite{Berarducci2018}]\label{prop:nr-decreases} We have: 
	\begin{enumerate}
		\item if $g$ is a proper non-zero truncation of $f$, then $\nr(g)<\nr(f)$.
		\item 	If $f=\sum_{i<\alpha}\m_i r_i$ is in normal form, then for every $i<\alpha$ we have $\nr(\m_i r_i) \leq \nr(f)$ and the inequality is strict if $i+1<\alpha$.
		\item $\nr(\pm e^{\gamma}) = \nr(\gamma)$ for all $\gamma\in \no^\uparrow$; 
		\item if $r\neq \pm 1$, then $\nr(e^\gamma r) = \nr(e^{\gamma})+1$ for all $\gamma\in \no^\uparrow$;
		\item For $x\in \no^*$, we have
		$\nr(x) = 0$ if and only if $x$ has the form $\pm \lambda^{\pm 1}$ where $\lambda$ is log-atomic. 
	\end{enumerate}
\end{prop}

\begin{cor}\label{cor:sameNR}
	If $f = \sum_{i<\alpha} e^{\gamma_i} r_i$ is in exponential normal form, then $\nr(\gamma_i) \leq \nr(f)$ for all $i<\alpha$, and $\nr(\gamma_i) < \nr(f)$ if $i+1<\alpha$. If $f$ has the same nested truncation rank of one of the exponents $\gamma_i$ in its exponential normal form, then
	$f = \sum_{j<i} e^{\gamma_j} r_j \pm e^{\gamma_i}$, with $\gamma_i$ being the smallest exponent. 
\end{cor}

\section{Branches in the tree of iterated exponents}
Given a surreal number $f=\sum_{i<\alpha} e^{\gamma_i}r_i$ written in exponential normal form we call each $\gamma_i\in \no^\uparrow$ an {\bf exponent} of $f$. Taking the exponents of the exponents and iterating the process we obtain a tree-like structure, where the children of $f$ are its exponents. We can then consider the infinite branches through this tree, as in the definition below. 
\begin{defn}
A {\bf branch} is a sequence $(B_n)_{n\in \N}$ of surreal numbers such that $B_{n+1}$ is an exponent of $B_n$ for each $n\in \N$. 
\end{defn}

 Notice that there are no branches $B$ starting with a real number $r= B_0\in \R$, because a non-zero real number $r = e^0 r$ has only zero as a possible exponent, and zero has no exponents at all. On the other hand any $f\in \no\setminus \R$ has at least one branch, because it has at least one non-zero exponent and such an exponent is necessarily in $\no^\uparrow \setminus \{0\} \subset \no \setminus \R$. 
A log-atomic number $x\in \no$ has exactly one branch, given by the iterated logarithms $\log_n(x)$ of $x$, which are themselves log-atomic elements. 
 
\begin{prop}\label{prop:max-branch}
Let $B$ be a branch such that $B_{n+1}$ is the greatest exponent of $B_n$ for every sufficiently large $n\in \N$. Then there is $m\in \N$ such that $B_m$ is log-atomic. 
\end{prop}

\begin{proof} Since there are no infinite decreasing sequences of ordinals, by \prettyref{cor:sameNR}, for all large enough $n\in \N$ we have $\nr(B_{n+1}) = \nr(B_n)$. This equation implies that $B_n  = \sum_{j<i} e^{\gamma_j} r_j \pm e^{B_{n+1}}$ with $B_{n+1}$ being the smallest exponent. On the other hand for $n$ sufficiently large $B_{n+1}$ is also the greatest exponent, so $B_n = \pm e^{B_{n+1}}$. Moreover, for $n>0$, $B_n$ is purely infinite (being an exponent), so if  $B_n=\pm e^{B_{n+1}}$, then $B_{n+1}$ is positive and we can write $B_{n+1} = e^{B_{n+2}}$. It follows that, for all large enough $n\in \N$, $B_n = e^{B_{n+1}}$, and therefore $B_n$ is log-atomic. 
\end{proof} 

\section{Transseries}
\begin{defn} Given a class $\Delta\subseteq \lambda_\no$ of log-atomic numbers, let $\R\langle \langle \Delta \rangle\rangle$ be the class of all $f\in \no$ such that every branch $B$ starting at $B_0 = f$ reaches an element of $\Delta$, i.e. there is $n\in \N$ with $B_n\in \Delta$. 
\end{defn}
\begin{prop}[\cite{Berarducci2018}]
$\R\langle \langle \Delta \rangle\rangle$ is the smallest subfield of $\no$ containing $\Delta$ and closed under $\exp$ and infinite sums of summable families. If $\log(\Delta)\subseteq \Delta$, then $\R\langle \langle \Delta \rangle\rangle$ is also closed under $\log$. The containment $\R\langle \langle \Delta \rangle\rangle \subset \no$ is always proper, because in $\no$ there are branches that do not reach any log-atomic number. 
\end{prop}
Taking $\Delta= \lambda_{-\N} = \{ \log_n(\omega) \mid n\in \N\}$, we obtain the field of {\bf omega-series}  $\R\langle \langle \lambda_{-\N}\rangle \rangle$, which will also be denoted $\R\langle \langle \omega \rangle\rangle$ as in \cite{Berarducci2018}. Thus $\R\langle \langle \omega \rangle\rangle$ is the smallest subfield of $\no$ containing $\R \cup \{\omega\}$ and closed under $\sum,\cdot, \exp,\log$. The field of omega-series is a proper class, like $\no$ itself, but it has interesting subfields which are sets, among which  an isomorphic copy of the field of transseries. 

\begin{defn} Define $\T= \T^{LE} \subset \T^{EL} \subset \R\langle \langle \lambda_{-\N} \rangle \rangle = \R\langle \langle \omega \rangle \rangle\subset \no$ as follows. 
	\begin{enumerate}
		\item $f \in \T^{EL}$ is there is $n\in \N$ such that for every branch $B$ starting at $B_0 = f$ we have $B_n \in \{\log_k(\omega)\mid k\in \N\}$. 
		\item $f\in \T$ if there are $m,n \in \N$ such that for every branch $B$ starting at $B_0 = f$ we have $B_n\in \{\log_k(\omega) \mid k\leq m\}$. 
	\end{enumerate}
\end{defn}
	Up to isomorphism, $\T$ coincides with the field of logarithmic-exponential series, or LE-series, considered by \citet{Dries1997, DriesMM2001}, with $\omega$ playing the role of the formal variable. The larger field $\T^{EL}\supset \T$ is the field of exponential-logarithmic series, or EL-series, of \citet{Kuhlmann2000}. Both $\T$ and $\T^{EL}$ are fields of transseries in the sense of \cite{Schmeling2001a}, but when we use the word {\bf transseries} without further specification we refer to the  elements of $\T$.  
The following example shows that containment $\T\subset \T^{EL}$ is proper. 
\begin{exa} 
	The monomial $\frac{1}{\prod_{n\in \N} \lambda_{-n}} = \frac{1}{\prod_{n\in \N} \log_n(\omega)}  = \exp(\sum_{n\in \N} -\log_n(\omega))$ is in $\T^{EL}$, but not in $\T$. 
\end{exa}
As exponential fields we have $\R_{\exp} \prec \T \prec \T^{EL} \prec \no$ (where ``$\prec$'' means ``elementary substructure''), however  
	$\T\not\prec T^{EL}$ when considered as differential fields with their natural derivation. Specifically, $\T$ is closed under integration (i.e. the derivative is surjective), while $\T^{EL}$ is not because there is no element in $\T^{EL}$ whose derivative is $\frac{1}{\prod_{n\in \N} \lambda_{-n}}$. Later we shall introduce a surjective strongly additive derivation $\partial: \no\to \no$ with $\partial \omega=1$ which extends the natural derivation on these transserial fields. 
\section{Paths}
In the previous section we have defined the notion of {\em branch}, using the iterated exponents in the exponential normal form. For the purpose of introducing a derivation, it is convenient to define a similar notion which however takes into account not only the exponents, but also the coefficients appearing in the normal form. This gives rise to the notion of {\em path}, defined in this section. 
\begin{defn} We define:
	\begin{enumerate}
		\item  A {\bf term} is a monomial multiplied by a non-zero real number. We shall write a term in the form $e^\gamma r$, or $re^\gamma$, where $\gamma\in \no^\uparrow$ and $r\in \R^*$. 
		\item We say that $e^\gamma r$ is a term of $f\in \no$ if it is one of the terms $e^{\gamma_i}r_i$ in its normal form $f= \sum_{i<\alpha} e^{\gamma_i}r_i$.
		\item A term $e^\gamma r$ is {\bf non-constant} if it does not belong to $\R$, namely $\gamma\neq 0$.
	\end{enumerate}  
\end{defn}

\begin{defn}\label{defn:path}
	A {\bf path} $P$ is a sequence $(P_n \mid n\in \N)$ of non-constant terms such that, for each $n\in \N$, if $P_n = r_n e^{\gamma_n}$, then  $P_{n+1}$ is one of the non-constant terms of $\gamma_n \in \no^\uparrow$, so we can write $$P_{n} = r_n \exp(x_{n+1} + P_{n+1} + \delta_{n+1})$$ where $\gamma_n = x_{n+1} + P_{n+1} + \delta_{n+1}\in \no^\uparrow$, $\delta_{n+1} \prec P_{n+1}$, and all the monomials of $x_{n+1}$ are $\succ P_{n+1}$. 
\end{defn}	
	
	The following rather surprising property about paths follows from the existence of the nested truncation rank. 
	
	\begin{prop}\label{prop:T4} Let $P$ be a path. In the above notations, for every sufficiently large $n\in \N$ we have $r_n = \pm 1$ and $\delta_{n+1}=0$, so $$P_n=\pm \exp({x_n + P_{n+1}}).$$
\end{prop}	
		\begin{proof}
		For every $n\in \N$, $\nr(P_{n+1}) \leq \nr(P_n)$, hence for every sufficiently large $n\in \N$ we have $\nr(P_{n+1}) = \nr(P_n)$.  For these values of $n\in \N$ we must have $P_n=\pm \exp({x_{n+1} + P_{n+1}})$. 
	\end{proof}
	
\prettyref{prop:T4} can be rephrased by saying that $\no$ satisfies {\bf axiom T4} in \citep{Schmeling2001a} and therefore $\no$ is a field of (generalized) transseries in the sense of that paper. 

\begin{rem}
The field of transseries $\T$ satisfies an even stronger property: every branch reaches a log-atomic number. This stronger axiom is called 
ELT4 in \cite{Kuhlmann2015} and it is also satisfied by the field of omega-series $\R\langle \langle \omega \rangle \rangle$, which is in fact the largest subfield of $\no$ satisfying the stronger property. The surreal field $\no$ contains {\em nested transseries} such as $y = \sqrt{\omega}+e^{\sqrt{\log(\omega)}+e^{\sqrt{\log_2(\omega)}+e^{\ldots}}}$ which violate ELT4, while still respecting T4. In this example, the element $y=y(\omega)$ is a solution of the functional equation
$y(\omega) = \sqrt{\omega}+e^{y(\log(\omega))}$. The perturbed equation $y(\omega) = \sqrt{\omega}+e^{y(\log \omega)} + \log \omega$ cannot have solutions in $\no$ since it would violate T4 (the example is taken from \cite{Aschenbrenner2017g}). 
\end{rem}
	
\begin{defn} \label{defn:pathof} Let $P$ be a path. 
	\begin{enumerate}
		\item We say that $P$ is a {\bf path of $f\in \no$}, if $P_0$ is a non-constant term of $f$. 
		Let 
		$$\mathcal{P}(f)= \{P \mid P \text{ is a path of } f\}.$$
		\item Given $f\in \no\setminus \R$, the {\bf dominant path} of $f$ is the unique path $P\in \mathcal{P}(f)$ such that $P_0 = r_0e^{\gamma_0}$ is the leading term of $f-f^\circ$ (where $f^{\circ}$ is the constant term of $f$, see \prettyref{defn:purely-infinite}) and for each $n$, if $P_n = r_n e^{\gamma_n}$, then $P_{n+1}$ is the leading term of $\gamma_n$. For the dominant path we can thus write $$P_n = r_n \exp(P_{n+1} + \delta_{n+1})$$ with $\delta_{n+1}\prec P_{n+1}$. 
	\end{enumerate}	
\end{defn} 

\begin{rem} We have: 
	\begin{enumerate}
		\item Given $f\in \no$, we have $\mathcal{P}(f) = \emptyset$ if and only if $f\in \R$;  
		\item Any log-atomic element $\lambda_x$ has exact one path $P$, given by $P_n = \log_n \lambda_x$. 
	\end{enumerate}
\end{rem}

%
%

\begin{prop}[\protect{\cite[Lemma 6.23]{Berarducci2018}}]\label{prop:dominant}
	Let $P$ be a dominant path. Then there is $n\in \N$ such that $P_n$ is log-atomic. 
\end{prop}
\begin{proof}
	By \prettyref{prop:T4} if $n\in \N$ is sufficiently large, $P_n=\pm \exp({x_n + P_{n+1}}).$ On the other hand, since $P$ is dominant, we also have $P_n = r_n \exp(P_{n+1} + \delta_{n+1}).$ Matching the two equations, we obtain 
	$P_n = \pm \exp(P_{n+1})$ for every sufficiently large $n$. This implies that $P_{n+1}$ is purely infinite, and therefore its exponent $P_{n+2}$ is positive. 
	It follows that for $n$ sufficiently large, $P_n = \exp(P_{n+1}),$ so $P_n$ is log-atomic. 
\end{proof}

\section{Surreal derivations}\label{surreal derivation}
A derivation on a field $K$ is a linear map $\partial: K \to K$ satisfying Leibniz's rule $\partial (xy) = x\partial y + y \partial x$. 

\begin{defn}\label{defn:surreal-derivation}
	A {\bf surreal derivation} is a derivation $\partial: \no\to \no$ satisfying: 
		\begin{enumerate}
		\item[(SD1)] if $x>\N$, then $\partial x > 0$;
		\item[(SD2)] $\ker(\partial) = \R$;
		\item[(SD3)] $\partial e^f = e^f \partial f$;
		\item[(SD4)]  if $(f_i)_{i<\alpha}$ is summable, then so is $(\partial f_i)_{i\in I}$ and $\partial (\sum_{i\in I} f_i) = \sum_{i\in I} \partial f_i.$
	\end{enumerate}
\end{defn}
The motivation for (SD1) comes from the theory of Hardy fields discussed in \prettyref{sec:Hardy}. 

To define a surreal derivation, a crucial problem is to decide its restriction to the log-atomic numbers. If we stipulate that $\partial \omega=1$, then necessarily $\partial \log_n(\omega) = \frac{1}{\prod_{m<n} \log_m(\omega)}$. This is a reasonable choice for the derivative of log-atomic numbers of the form $\log_n(\omega)$ because it corresponds to the elementary rule of calculus
$\frac{d\log_n(\x)}{d\x} = \frac{1}{\prod_{m<n} \log_m(\x)}$, with $\omega$ in the role of the formal variable $\x$. In the general case we have the following result.

\begin{thm}[\protect{\cite[Thm. 6.30]{Berarducci2018}}]  \label{thm:derivation} 
	There is a (possibly not unique) surjective surreal derivation $\partial: \no\to \no$ whose restriction to the log-atomic numbers is given by
	\begin{equation*}  \leqno{\hspace{4pt} \text{\em (SD5)}} \quad 
	\partial \lambda_{-x} = \frac{\prod\limits_{n\in \N} \lambda_{-x-n}}{\prod\limits_{\beta}\lambda_{-\beta}} 
	\end{equation*}
	where $x\in \no$ and $\beta$ ranges over the ordinals $\leq x+m$ for some $m\in \N$.
	Note that when $x$ is an ordinal, the condition on $\beta$ becomes $\beta < x + \omega$. As special cases of (SD5) we have
	\begin{enumerate}
		\item $\partial \omega = 1$; 
		\item $\partial \log_n(\omega) = \frac{1}{\prod_{m<n} \log_m(\omega)}$;
		\item $\partial \lambda_{-\alpha} = \frac{1}{\prod_{\beta<\alpha} \lambda_{-\beta}}\;$ for every ordinal $\alpha$\\ (where $\beta$ ranges over the ordinals $<\alpha$). 
	\end{enumerate}
	
\end{thm}

\begin{rem}\label{rem:partialomega} 
The  present form of (SD5) appears in \citep{Aschenbrenner2019} and is a notational simplification of the equivalent formula in \cite[Def. 6.7]{Berarducci2018}.
To prove that (SD5) implies (1)--(3) we reason as follows. When
 $\alpha$ is an ordinal, the fraction in (SD5) simplifies as 
 $\partial \lambda_{-\alpha} = \frac{\prod\limits_{n\in \N} \lambda_{-\alpha-n}}{\prod\limits_{\beta< \alpha+\omega}\lambda_{-\beta}} = \frac{1}{\prod_{\beta<\alpha} \lambda_{-\beta}}
 $. 
    Recalling that $\lambda_{-n} = \log_n(\omega)$ for $n\in \N$, we obtain, 
	$\partial \log_n(\omega) = \frac{1}{\prod_{m<n} \log_m(\omega)}$,  so in particular $\partial \omega = 1$. 
	
	To justify (SD5) in the general case we argue as follows. It can be shown that (SD1)--(SD4) imply that
$\log \partial x - \log \partial y \prec x-y \preceq \max \{ x, y\}$ for all $x,y\in \no^{>N}$ such that $x-y >\N$ \cite[Prop. 6.5]{Berarducci2018}. Accordingly, we define a {\bf prederivation} as a map $D:\lambda_\no\to \no^{>0}$ satisfying
$\partial e^x = e^x D(x) $ and $\log D(x) - \log D(y) \prec \max \{x,y\}$ for $x>y$ in $\lambda_\no$ (recall that the elements of $\lambda_\no$ and their differences are positive infinite). 
Given two prederivations, we say that one is simpler than the other if it takes a simpler value on any element of minimal simplicity where they differ. It turns out that the {\bf simplest prederivation} is given by  (SD5) \cite[Thm. 9.6]{Berarducci2018}. 
\end{rem}

 
To prove \prettyref{thm:derivation} we need to show how to extend a prederivation $D:\lambda_\no\to \no$ to a surreal derivation. To this aim we observe that {\em every} surreal derivation must satisfy
\begin{equation*}\label{eq:recursion}
\partial (\sum_{i<\alpha} e^{\gamma_i}r_i) = \sum_{i<\alpha} e^{\gamma_i}r_i \partial \gamma_i,
\end{equation*}
so the derivative of $f= \sum_{i<\alpha} e^{\gamma_i}r_i$ is a sum of contributions (possibly with cancellations) given by the {\em terms} $e^{\gamma_i}r_i$ multiplied by the derivative of the corresponding {\em exponents} $\gamma_i$. An iteration of the process involves the terms appearing in the exponents of $f$, then the exponents of the exponents, and so on. 
This process never ends, unless we start with a real number. Eventually, we are led to write $\partial f$ as a sum of contributions $\partial_D(P)\in \no$, one for each path $P\in \mathcal{P}(f)$, as in the following definition (adapted to the surreals from a similar construction in \cite{Schmeling2001a}).  

\begin{defn} \label{defn:pathD} Let $D:\lambda_\no\to \no$ be a prederivation and let $P=(P_n\mid \n\in \N)$ be a path. We define its {\bf path-derivative} $\partial_D(P)\in \no$ as follows. 
	If one of the terms $P_n$ of $P$ is a log-atomic number, we put 
	\begin{equation*}\label{eq:DP}
	\partial_D(P) = \left(\prod_{m<n} P_m\right) D(P_n)
	\end{equation*} 
	and we observe that the above definition does not depend on the choice of $n$ (because if $P_n$ is log-atomic $D(P_{n}) = P_{n}D(P_{n+1})$). 
	On the other hand, if $P$ does not reach a log-atomic number, we define its path-derivative $\partial_D(P)$ to be zero. Under the simplifying assumption that $D$ takes values in $\M \R^*$ \footnote{The assumption can probably be relaxed, but we do not insist on this point since we are mainly interested in the simplest prederivation, whose values are in $\M$.}, it can be proved that the family $(\partial_D(P)\mid P\in \mathcal{P}(f))$ is summable \cite[Prop. 6.20, Thm. 6.32]{Berarducci2018}, so we can define 
	\begin{equation*}
	\partial f = \sum_{P\in \mathcal{P}(f)} \partial_D(P)
	\end{equation*}
	 and we call $\partial$ the {\bf derivation induced by the prederivation $D$}. 
\end{defn}

\begin{prop}[\protect{\cite[Thm. 6.32]{Berarducci2018}}] Given a prederivation $D:\lambda_\no \to \M \R^*$, the map $\partial:\no\to \no$ defined as above is a surreal derivation.
\end{prop}
\begin{proof}[Sketch of Proof] 
	We omit the verification that $\partial$ is a strongly additive derivation compatible with $\exp$ and we only make a comment on the proof that $\ker \partial = \R$. Clearly $\R\subseteq \ker \partial$ because for $f\in \R$ we have $\mathcal{P}(f) = \emptyset$.  On the other hand if $f\nin \R$ and $P\in \mathcal{P}(f)$ is the dominant path, then $P$ reaches a log-atomic number by \prettyref{prop:dominant}, so $\partial_D(P)\neq 0$. Moreover, in the formula $\partial f = \sum_{P\in \mathcal{P}(f)} \partial_D(P)$, the contribution of the dominant path  dominates the contributions of the other paths in $\mathcal{P}(f)$, so it cannot be cancelled 
\end{proof}

To complete the proof of \prettyref{thm:derivation} one needs to establish the following proposition, for the proof of which we refer to the original paper. 

\begin{prop}[\protect{\cite[Prop. 7.6]{Berarducci2018}}]\label{prop:simplest}
	The derivation $\partial:\no\to \no$ induced by the simplest pre-derivation is surjective, hence it is as prescribed in \prettyref{thm:derivation}. 
\end{prop}

In the rest of the paper, we let $\partial:\no\to \no$ be the surreal derivation induced by the simplest pre-derivation. If $\partial f = g$, 
we say that $f$ is an {\bf integral} of $g$. Since $\ker \partial = \R$, the integral is defined up to an additive real constant. 

\begin{rem}
If we modify (SD5) putting $\partial \lambda_{-x} = \prod_{n\in \N} \lambda_{-x-n}$, we obtain a different surreal derivation which however is not surjective as there would be no integral of $1$ (see \cite[Def. 6.6]{Berarducci2018}). Despite the fact that the formula $\partial \lambda_{-x} = \prod_{n\in \N} \lambda_{-x-n}$ looks intuitively simpler than (SD5), the pre-derivation given by (SD5) is in fact simpler according to our definitions. 
\end{rem}

\section{Hardy fields}\label{sec:Hardy}
In this section we discuss the connections between surreal numbers and Hardy fields. 

\begin{defn}[See \citep{Bourbaki1976}]
	A {\bf Hardy field} is a set of germs at $+\infty$ of real valued differentiable functions on positive half-lines of $\R$ that form a field under the usual addition and multiplication of germs and it is closed under differentiation of germs.
\end{defn}	
	There is a large literature on Hardy fields, see for instance the pioneering work of Hardy \cite{Hardy1910} and Rosenlicht  \cite{Rosenlicht1983,Rosenlicht1983a,Rosenlicht1987}. 
	If the germ of $f$ belongs to a Hardy field, then $f$ is eventually always positive, or negative, or zero (we cannot have $\sin(\x)$), so each Hardy field is an ordered field: its positive elements are the germs of the eventually positive functions. 
	\begin{exa}
	The field $\R(\x)$ of rational functions is a Hardy field (with the order induced by $\x>\R$). The set of germs at $+\infty$ of functions definable in an o-minimal expansion of the reals is a Hardy field. 
%
	\end{exa}
	
The notion of H-fields is an algebraic counterpart of the notion of Hardy field. 
	
	\begin{defn}[\cite{Aschenbrenner2002,Aschenbrenner2005a}]
	An {\bf H-field} is an ordered field $K$ with a derivation $\partial: K \to K$ satisfying
	\begin{enumerate}
		\item If $f>c$ for every $c\in \ker(\partial)$, then $\partial f >0$; 
		\item
		If $|f|$ is bounded by some constant $c\in \ker \partial$, then $f$ is equal to a constant plus an element smaller in absolute value than any positive constant). 
	\end{enumerate}
	We say that an H-field has {\bf small derivation} if whenever $|f|$ is smaller than any constant, so is $|\partial f|$. 
\end{defn}
Note that if an H-field $K \supseteq \R$ has a derivation $\partial$ with $\ker \partial = \R$ and there is an element $f\in K$ with $\partial f = 1$, then the derivation is small. If $\ker \partial = \R$, then (1) takes the form $f>\N\implies \partial f >0$, which coincides with property (SD1) in the definition of surreal derivation (\prettyref{defn:surreal-derivation}).  

\begin{exa}
Any Hardy field is a H-field. Other examples include: the field of Laurent series, the field of Puiseux series, the field $\T$ of transseries, the field $\no$ of surreal numbers (with the derivation of \prettyref{prop:simplest}).   
\end{exa}
We have already mentioned the fact that every linear order (whose domain is a set) embeds in $\no$. Similarly, it can be proved that many ordered algebraic structures embed nicely in $\no$ \cite{Ehrlich2001a,Ehrlich2020}. In a similar spirit, we have the following result. 

\begin{thm}[\cite{Aschenbrenner2019}]\label{thm:embeddingH}
	Every Hardy field extending $\R$ (and more generally every H-field with a small derivation and constant field $\R$) admits an embedding into $(\no, \partial)$ as a differential field. 
\end{thm}

\citet*{Aschenbrenner2017} proved that every system of algebraic differential equations (in several variables) over the field $\T$ of transseries which has a solution in a larger H-field, has already a solution in $\T$. Moreover there is an algorithm which decides when this happens. Indeed the authors of \cite{Aschenbrenner2017} showed that the complete (first--order) theory of $\T$ as a differential field is recursively axiomatisable, hence decidable. Combining this with \prettyref{thm:embeddingH}, we deduce that $\no$ is closed under a large class of differential equations compatible with the theory of H-fields. For instance, if $P(y)$ is a polynomial over $\no$, there is some $f\in \no$ such that $\partial f = P(f)$ (it suffices to observe that the corresponding result holds in $\T$ and the coefficients of $P(y)$ lie in some small H-field contained in $\no$, where ``small'' means that the domain is a set). 

\section{Composition}
	Given two formal power series $f(\x)$ and $g(\x)$ such that $g(\x)$ has no constant term, we can define the composition $f(g(\x))$ simply by substituting all occurrences of $\x$ in $f(\x)$ with $g(\x)$ and expanding the resulting expression. In a similar way, we can define the composition $f\circ g$ of two transseries $f,g\in \T$ provided $g>\N$. More generally, given two omega-series $f,g\in \R\langle \langle \omega \rangle\rangle$ with $g>\N$ we can define $f\circ g$ replacing each occurrence of $\omega$ in $f$ with $g$ and putting the resulting expression in normal form, where the hypothesis $g>\N$ is needed to ensure the summability of the development. We can even allow the second argument of the composition to be a surreal number, as in the following result. 
	
	\begin{thm}[\cite{Berarducci2019}]\label{thm:composition} There is a unique map  
	$\circ: \R\langle \langle \omega \rangle \rangle \times \no^{>\N} \to \no$ satisfying   
		\begin{enumerate}
			\item $r\circ x = r$ if $r\in \R$;
			\item $\omega\circ x = x$;
			\item $(\sum_{i\in I}f_i) \circ x = \sum_{i\in I} (f_i\circ x)$;
			\item $\exp(f)\circ x = \exp(f\circ x)$.
		\end{enumerate}	
	\end{thm}
	Note that (3) also implies $\log(h)\circ x = \log(h\circ x)$, so in particular $\log_n(\omega) \circ x = \log_n(x)$. Since every branch of an element of $\R\langle \langle \omega \rangle \rangle$ reaches a log-atomic number of the form $\log_n(\omega)$, points (1)--(4) determine the value of the composition $f\circ x$ for $f\in \R\langle \langle \omega \rangle \rangle$ and $x\in \no^{>\N}$, thus proving the uniqueness part of the theorem. The existence part is more complicated, as one needs to verify inductively that if $f_i\circ x$ has been defined for every $i\in I$ and the family $(f_i)_{i\in I}$ is summable, then also $(f_i\circ x)_{i\in I}$ is summable.
	
	 The associativity property $(f\circ g)\circ x = f\circ (g\circ x)$ holds whenever it makes sense, i.e. when both $f$ and $g$ are omega-series (while $x\in \no^{>\N}$). It follows that we can interpret an omega-series $f$ (and in particular a transseries), as  
	a surreal function $\no^{>\N} \to \no$ sending $x$ to $f\circ x$. We also write $f(x)$ instead of $f\circ x$.
	
	 Since $\no$ is an ordered field, it makes sense to ask whether the function $x\mapsto f\circ x$ is differentiable. Point (1) of the following proposition shows that the answer is positive and the derivative of this function coincides with the the function $x\mapsto (\partial f) \circ x$. 
	 \begin{prop}[\cite{Berarducci2019}] \label{prop:compatible}For $f,g\in \R\langle\langle \omega \rangle\rangle$ with $g>\N$, $x,y\in \no^{>\N}$ and $\eps\in \no$, we have: 
	 	\begin{enumerate}
	 		\item $\partial f \circ x = \lim\limits_{\eps\to 0} \frac{f\circ (x+\eps)-f\circ x}{\eps}$;
	 		\item $\partial (f \circ g) = (\partial f \circ g)\partial g$;
	 		\item $f\circ (b+\eps) = \sum_{n\in \N} \frac{1}{n!}(\partial^{(n)} f \circ b)\eps^n$, provided $\eps$ is sufficiently small;
	 		\item if $\partial f = 0$, then $f\circ x = f$;
	 		\item if $\partial f>0$ and $x<y$, then $f\circ x < f \circ y$. 
	 	\end{enumerate}
%
	\end{prop}
A natural question is whether we can extend the composition $f,x\mapsto f\circ x$ allowing both $f$ and $x$ to be surreal numbers (with $x>\N$). We shall return to this problem in the next section. 

\section{Conclusions}
The following table compares $\no$ with other rings and fields of formal series which can be embedded in $\no$ (with the ordinal $\omega$ playing the role of a formal variable). 
\begin{center}
	Power series $\;\subset \;$	Laurent series $\; \subset \; $ Puiseux series $\;\subset \;$ $\T$  $\;\subset \;$ $\no$. \end{center}

\smallskip

\begin{center}
		\begin{tabular}{ l| c | c| c| c|  c|}
		\cline{2-6}
		& Power series  & Laurent &   Puiseux & Transseries &  Surreals \\ 
		\hline
		\multicolumn{1}{|l|}{Ordered ring} & \checkmark & \checkmark &  \checkmark & \checkmark & \checkmark \\ \hline
		\multicolumn{1}{|l|}{Ordered field} &  & \checkmark &  \checkmark & \checkmark & \checkmark \\ \hline
		\multicolumn{1}{|l|}{Derivation} &  \checkmark & \checkmark & \checkmark & \checkmark & \checkmark \\ \hline
		\multicolumn{1}{|l|}{	Integrals} & \checkmark &  &  & \checkmark & \checkmark \\ 	\hline
		\multicolumn{1}{|l|}{ Infinite sums}  & \checkmark & \checkmark & restricted & restricted & \checkmark \\ 
		\hline
		\multicolumn{1}{|l|}{ exp and log }  &  &  &   & \checkmark & \checkmark \\ 
		\hline
		\multicolumn{1}{|l|}{ Composition}  &  \checkmark& \checkmark & \checkmark  & \checkmark & ? \\ 
		\hline	
	\end{tabular} 
\end{center} 
\bigskip

In the table, 
``Integrals'' means that the derivation is surjective
and ``Restricted'' means that we do not consider all the possible sums of summable families, as illustrated by the following examples: 
\begin{itemize}
	\item $\sum_{n\in \N} \omega^{1/n}$ is a transseries (in $\omega$), but not a Puiseux series; 
	\item 
	$\sum_{n\in \N} \log_n(\omega)$ is a surreal, but not a transseries. 
\end{itemize}

It is an open problem whether there is an associative composition $\circ : \no\times \no^{>\N}\to \no$ satisfying the properties in \prettyref{thm:composition}. This
is connected to the problem of representing $\no$ as a field of ``hyperseries'' \cite{Aschenbrenner2017g}. 
Given a composition and a surreal derivation, we may require that they satisfy the compatibility relations expressed by \prettyref{prop:compatible}. In this case the derivation is actually definable in terms of the composition using the formula $\partial f = \lim\limits_{\eps\to 0} \frac{f\circ (\omega +\eps)-f\circ \omega}{\eps}$. 

In \cite[Thm. 8.4]{Berarducci2019} it is shown that the surreal derivation considered in this paper (the derivation defined in \prettyref{prop:simplest}) is not compatible with a composition. 
There could however be other derivations, obtained by modifying (SD5) in \prettyref{thm:derivation}, which are compatible with a composition. 
In \cite{VandenDries2019a} it is shown that, up to isomorphism, the derivation of \prettyref{prop:simplest} is the unique one satisfying some natural properties. Those properties however do not include the fact that the derivation distributes over infinite sums. The problem whether there is a composition and a compatible ``better derivation'' is therefore still unresolved. 

\bibliographystyle{plainnat}

\begin{thebibliography}{46}
%
	\bibitem[Alling(1987)]{Alling1987}
	Norman~L. Alling.
	\newblock \emph{{Foundations of Analysis over Surreal Number Fields}}, volume
	141 of \emph{North-Holland Mathematics Studies}.
	\newblock North-Holland Publishing Co., Amsterdam, 1987.	\newblock ISBN 0-444-70226-1.
	
	\bibitem[Aschenbrenner and van~den Dries(2002)]{Aschenbrenner2002}
	Matthias Aschenbrenner and Lou van~den Dries.
	\newblock {$H$-fields and their Liouville extensions}.
	\newblock \emph{Math. Zeitschrift}, 242\penalty0 (3):\penalty0 543--588, 2002.
	\newblock \doi{10.1007/s002090000358}.
	
	\bibitem[Aschenbrenner and van~den Dries(2005)]{Aschenbrenner2005a}
	Matthias Aschenbrenner and Lou van~den Dries.
	\newblock {Liouville closed H-fields}.
	\newblock \emph{J. Pure Appl. Algebr.}, 197\penalty0 (1-3):\penalty0 83--139,
	may 2005.
	\newblock \doi{10.1016/j.jpaa.2004.08.009}.
	
	\bibitem[Aschenbrenner et~al.(2017)Aschenbrenner, van~den Dries, and van~der
	Hoeven]{Aschenbrenner2017}
	Matthias Aschenbrenner, Lou van~den Dries, and Joris van~der Hoeven.
	\newblock \emph{{Asymptotic Differential Algebra and Model Theory of
			Transseries}}.
	\newblock Annals of Mathematical Studies. Princeton University Press,
	Princeton, dec 2017.
	\newblock ISBN 9781400885411.
	\newblock \doi{10.1515/9781400885411}.
	
	\bibitem[Aschenbrenner et~al.(2019{\natexlab{a}})Aschenbrenner, van~den Dries,
	and van~der Hoeven]{Aschenbrenner2017g}
	Matthias Aschenbrenner, Lou van~den Dries, and Joris van~der Hoeven.
	\newblock {On numbers, germs, and transseries}.
	\newblock In \emph{Proc. Int. Congr. Math. (ICM 2018)}, pages 1--23. WORLD
	SCIENTIFIC, may 2019{\natexlab{a}}.
	\newblock \doi{10.1142/9789813272880_0042}.
	
	\bibitem[Aschenbrenner et~al.(2019{\natexlab{b}})Aschenbrenner, {Van Den
		Dries}, and {Van Der Hoeven}]{Aschenbrenner2019}
	Matthias Aschenbrenner, Lou {Van Den Dries}, and Joris {Van Der Hoeven}.
	\newblock {The surreal numbers as a universal H-field}.
	\newblock \emph{J. Eur. Math. Soc.}, 21\penalty0 (4):\penalty0 1179--1199,
	2019{\natexlab{b}}.
	\newblock \doi{10.4171/JEMS/858}.
	
	\bibitem[Bagayoko and Hoeven(2019)]{Bagayoko2019}
	Vincent Bagayoko and Joris Van~Der Hoeven.
	\newblock {Surreal substructures}.
	\newblock \emph{hal-02151377}, 2019.
	
	\bibitem[Berarducci and Mamino(2019)]{Berarducci2019a}
	Alessandro Berarducci and Marcello Mamino.
	\newblock {Asymptotic analysis of Skolem's exponential functions}.
	\newblock \emph{arXiv}, \penalty0 (November):\penalty0 1--23, 2019.
	\newblock URL \url{http://arxiv.org/abs/1911.07576}.
	
	\bibitem[Berarducci and Mantova(2018)]{Berarducci2018}
	Alessandro Berarducci and Vincenzo Mantova.
	\newblock {Surreal numbers, derivations and transseries}.
	\newblock \emph{J. Eur. Math. Soc.}, 20\penalty0 (2):\penalty0 339--390, jan
	2018.
	\newblock \doi{10.4171/JEMS/769}.
	
	\bibitem[Berarducci and Mantova(2019)]{Berarducci2019}
	Alessandro Berarducci and Vincenzo Mantova.
	\newblock {Transseries as germs of surreal functions}.
	\newblock \emph{Trans. Am. Math. Soc.}, 371\penalty0 (5):\penalty0 3549--3592,
	dec 2019.
	\newblock ISSN 0002-9947.
	\newblock \doi{10.1090/tran/7428}.
	
		\bibitem[Berarducci et~al.(2018)Berarducci, Kuhlmann, Mantova, and
		Matusinski]{Berarducci2018b}
		Alessandro Berarducci, Salma Kuhlmann, Vincenzo Mantova, and Micka{\"{e}}l
		Matusinski.
		\newblock {Exponential fields and Conway's omega-map}.
		\newblock \emph{Proc. Amer. Math. Soc, to Appear}, pages 1--15, 2018.
		\newblock URL \url{http://arxiv.org/abs/1810.03029}.
		
		\bibitem[Bourbaki(1976)]{Bourbaki1976}
		Nicolas Bourbaki.
		\newblock \emph{{Fonctions d'une variable r\'{e}elle: Th\'{e}orie
				\'{e}l\'{e}mentaire}}.
		\newblock El\'{e}ments de math\'{e}matique. Diffusion C.C.L.S., Paris, 1976.
		\newblock ISBN 9783540340386.
		
		\bibitem[Conway(1976)]{Conway76}
		John~H. Conway.
		\newblock \emph{{On number and games}}, volume~6 of \emph{London Mathematical
			Society Monographs}.
		\newblock Academic Press, London, 1976.
		\newblock ISBN 0-12-186350-6.
		
		\bibitem[Costin et~al.(2015)Costin, Ehrlich, and Friedman]{Costin2015}
		Ovidiu Costin, Philip Ehrlich, and Harvey~M. Friedman.
		\newblock {Integration on the Surreals: a Conjecture of Conway, Kruskal and
			Norton}.
		\newblock may 2015.
		\newblock URL \url{http://arxiv.org/abs/1505.02478}.
		
		\bibitem[Denef and van~den Dries(1988)]{Denef1988}
		J~Denef and Lou van~den Dries.
		\newblock p-adic and real subanalytic sets.
		\newblock \emph{Ann. Math.}, 128\penalty0 (1):\penalty0 79--138, 1988.
		\newblock URL \url{http://www.jstor.org/stable/1971463}.
		
		\bibitem[Dulac(1923)]{Dulac1923}
		H.~Dulac.
		\newblock {Sur les cycles limites}.
		\newblock \emph{Bull. la Soci{\'{e}}t{\'{e}} Math{\'{e}}matique Fr.},
		51:\penalty0 45--188, 1923.
		\newblock \doi{10.24033/bsmf.1031}.
		
		\bibitem[Ecalle(1992)]{Ecalle1992}
		Jean Ecalle.
		\newblock \emph{{Introduction aux fonctions analysables et preuve constructive
				de la conjecture de Dulac}}.
		\newblock Actualit\'es Math\'ematiques. Hermann, Paris, 1992.
		\newblock ISBN 2-7056-6199-9.
		
		\bibitem[Ehrlich(2001)]{Ehrlich2001a}
		Philip Ehrlich.
		\newblock {Number Systems with Simplicity Hierarchies : A Generalization of
			Conway's Theory of Surreal Numbers}.
		\newblock \emph{J. Symb. Log.}, 66\penalty0 (3):\penalty0 1231--1258, 2001.
		\newblock URL \url{http://journals.cambridge.org/abstract_S0022481200010604}.
		
		\bibitem[Ehrlich(2011)]{Ehrlich2011}
		Philip Ehrlich.
		\newblock {Conway names, the simplicity hierarchy and the surreal number tree}.
		\newblock \emph{J. Log. Anal.}, 1\penalty0 (January):\penalty0 1--26, 2011.
		\newblock \doi{10.4115/jla.2011.3.1}.
		
		\bibitem[Ehrlich and Kaplan(2020)]{Ehrlich2020}
		Philip Ehrlich and Elliot Kaplan.
		\newblock {Surreal ordered exponential fields}.
		\newblock pages 1--25, 2020.
		\newblock URL \url{http://arxiv.org/abs/2002.07739}.
		
		\bibitem[Gonshor(1986)]{Gonshor1986}
		Harry Gonshor.
		\newblock \emph{{An introduction to the theory of surreal numbers}}.
		\newblock London Mathematical Society Lecture Notes Series. Cambridge
		University Press, Cambridge, 1986.
		\newblock \doi{10.1017/CBO9780511629143}.
		
		\bibitem[Hahn(1907)]{Hahn1907}
		Hans Hahn.
		\newblock {{\"{U}}ber die nichtarchimedischen Gr{\"{o}}ssenszsteme}.
		\newblock \emph{Sitz. der K. Akad der Wiss., Math Nat. KL.}, 116\penalty0
		(IIa):\penalty0 601--655, 1907.
		
		\bibitem[Hardy(1910)]{Hardy1910}
		G.~H. Hardy.
		\newblock \emph{{Orders of infinity,The `infinit\"arcalc\"ul' of Paul du
				Bois-Reymond}}.
		\newblock Cambridge University Press, 1910.
		
		\bibitem[Kaplan(2020)]{Kaplan2020}
		Elliot Kaplan.
		\newblock {Model completeness for the differential field of transseries with
			exponentiation}.
		\newblock \emph{ArXiv:2004.1495v2}, pages 1--27, May 2020.
		\newblock URL \url{http://arxiv.org/abs/2004.14957}.
		
		
		\bibitem[Kaplansky(1942)]{Kaplansky1942}
		Irvin Kaplansky.
		\newblock {Maximal fields with valuations}.
		\newblock \emph{Duke Math. J.}, 9:\penalty0 303--321, 1942.
		
		\bibitem[Kuhlmann et~al.(1997)Kuhlmann, Kuhlmann, and Shelah]{Kuhlmann1997}
		Franz-Viktor Kuhlmann, Salma Kuhlmann, and Saharon Shelah.
		\newblock {Exponentiation in power series fields}.
		\newblock \emph{Proc. Am. Math. Soc.}, 125\penalty0 (11):\penalty0 3177--3183,
		1997.
		\newblock \doi{10.1090/S0002-9939-97-03964-6}.
		
		\bibitem[Kuhlmann(2000)]{Kuhlmann2000}
		Salma Kuhlmann.
		\newblock \emph{{Ordered Exponential Fields}}, volume~12 of \emph{Fields
			Institute Monographs}.
		\newblock American Mathematical Society, Providence, Rhode Island, 2000.
		\newblock ISBN 0-8218-0943-1.
		
		\bibitem[Kuhlmann and Matusinski(2015)]{Kuhlmann2015}
		Salma Kuhlmann and Micka{\"{e}}l Matusinski.
		\newblock {The Exponential-Logarithmic Equivalence Classes of Surreal Numbers}.
		\newblock \emph{Order}, 32\penalty0 (1):\penalty0 53--68, mar 2015.
		\newblock \doi{10.1007/s11083-013-9315-3}.
		
		\bibitem[Kuhlmann and Shelah(2005)]{Kuhlmann2005}
		Salma Kuhlmann and Saharon Shelah.
		\newblock {$\kappa$}-bounded exponential-logarithmic power series fields.
		\newblock \emph{Ann. Pure Appl. Log.}, 136\penalty0 (3):\penalty0 284--296, nov
		2005.
		\newblock ISSN 01680072.
		\newblock \doi{10.1016/j.apal.2005.04.001}.
		
		\bibitem[Kuhlmann and Tressl(2012)]{Kuhlmann2012d}
		Salma Kuhlmann and Marcus Tressl.
		\newblock {Comparison of exponential-logarithmic and logarithmic-exponential
			series}.
		\newblock \emph{Math. Log. Q.}, 58\penalty0 (6):\penalty0 434--448, nov 2012.
		\newblock ISSN 09425616.
		\newblock \doi{10.1002/malq.201100113}.
		
		\bibitem[Macintyre and Wilkie(1996)]{Macintyre1996a}
		Angus Macintyre and Alex~J. Wilkie.
		\newblock {On the decidability of the real exponential field}.
		\newblock In Odifreddi, editor, \emph{Kreiseliana About Around Georg Kreisel},
		pages 441--467. A K Peters, 1996.
		\newblock ISBN 156881061X.
		
		\bibitem[Mantova and Matusinski(2017)]{Mantova2017a}
		Vincenzo Mantova and Micka{\"{e}}l Matusinski.
		\newblock {Surreal numbers with derivation, Hardy fields and transseries: a
			survey}.
		\newblock In \emph{Ordered Algebr. Struct. Relat. Top.}, volume 697 of
		\emph{Contemporary Mathematics}, pages 265--290. Amer. Math. Soc.,
		Providence, RI, 2017.
		\newblock \doi{10.1090/conm/697/14057}.
		
		\bibitem[Neumann(1949)]{Neumann1949}
		Bernhard~Hermann Neumann.
		\newblock {On ordered division rings}.
		\newblock \emph{Trans. Amer. Math. Soc}, 66\penalty0 (1):\penalty0 202--252,
		1949.
		
		\bibitem[Ressayre(1993)]{Ressayre1993}
		Jean-Pierre Ressayre.
		\newblock {Integer parts of real closed exponential fields (extended
			abstract)}.
		\newblock In Peter Clote and J.~Kraj\'\i\v{c}ek, editors, \emph{Arith. Proof
			Theory, Comput. Complex. (Prague, 1991)}, volume~23 of \emph{Oxford Logic
			Guides}, pages 278--288. Oxford University Press, New York, 1993.
		\newblock ISBN 978-0-19-853690-1.
		
		\bibitem[Rosenlicht(1983{\natexlab{a}})]{Rosenlicht1983}
		Maxwell Rosenlicht.
		\newblock {The rank of a Hardy field}.
		\newblock \emph{Trans. Am. Math. Soc.}, 280\penalty0 (2):\penalty0 659--671,
		1983{\natexlab{a}}.
		\newblock \doi{10.2307/1999639}.
		
		\bibitem[Rosenlicht(1983{\natexlab{b}})]{Rosenlicht1983a}
		Maxwell Rosenlicht.
		\newblock {Hardy fields}.
		\newblock \emph{J. Math. Anal. Appl.}, 93\penalty0 (2):\penalty0 297--311, may
		1983{\natexlab{b}}.
		\newblock \doi{10.1016/0022-247X(83)90175-0}.
		
		\bibitem[Rosenlicht(1987)]{Rosenlicht1987}
		Maxwell Rosenlicht.
		\newblock {Growth properties of functions in Hardy fields}.
		\newblock \emph{Trans. Am. Math. Soc.}, 299\penalty0 (1):\penalty0 261--261,
		jan 1987.
		\newblock \doi{10.2307/2000493}.
		
		\bibitem[Schmeling(2001)]{Schmeling2001a}
		Michael~Ch. Schmeling.
		\newblock \emph{{Corps de transs{\'{e}}ries}}.
		\newblock Phd, Universit{\'{e}} de Paris 7, 2001.
		\newblock URL \url{http://cat.inist.fr/?aModele=afficheN&cpsidt=14197291}.
		
		\bibitem[Skolem(1956)]{Skolem1956}
		Thoralf Skolem.
		\newblock {An ordered set of arithmetic functions representing the least
			{$\epsilon$}-number}.
		\newblock \emph{Nor. Vid. Selsk. Forh., Trondheim}, 29:\penalty0 54--59, 1956.
		
		\bibitem[van~den Dries and Ehrlich(2001)]{DriesE2001}
		Lou van~den Dries and Philip Ehrlich.
		\newblock {Fields of surreal numbers and exponentiation}.
		\newblock \emph{Fundam. Math.}, 167\penalty0 (2):\penalty0 173--188, 2001.
		\newblock \doi{10.4064/fm167-2-3}.
		
		\bibitem[van~den Dries and Ehrlich(2019)]{VandenDries2019a}
		Lou van~den Dries and Philip Ehrlich.
		\newblock {Homogeneous universal $H$-fields}.
		\newblock \emph{Proc. Am. Math. Soc.}, 147\penalty0 (5):\penalty0 2231--2234,
		2019.
		\newblock \doi{10.1090/proc/14424}.
		
		\bibitem[van~den Dries and Levitz(1984)]{Dries1984}
		Lou van~den Dries and Hilbert Levitz.
		\newblock {On Skolem's exponential functions below $2^{2^x}$}.
		\newblock \emph{Trans. Am. Math. Soc.}, 286\penalty0 (1):\penalty0 339--349,
		1984.
		
		\bibitem[van~den Dries et~al.(1994)van~den Dries, Macintyre, and
		Marker]{Dries1994}
		Lou van~den Dries, Angus Macintyre, and David Marker.
		\newblock {The elementary theory of restricted analytic fields with
			exponentiation}.
		\newblock \emph{Ann. Math.}, 140\penalty0 (1):\penalty0 183--205, 1994.
		\newblock URL \url{http://www.jstor.org/stable/2118545}.
		
		\bibitem[van~den Dries et~al.(1997)van~den Dries, Macintyre, and
		Marker]{Dries1997}
		Lou van~den Dries, Angus Macintyre, and David Marker.
		\newblock {Logarithmic-Exponential Power Series}.
		\newblock \emph{J. London Math. Soc.}, 56\penalty0 (3):\penalty0 417--434, dec
		1997.
		\newblock \doi{10.1112/S0024610797005437}.
		
		\bibitem[van~den Dries et~al.(2001)van~den Dries, Macintyre, and
		Marker]{DriesMM2001}
		Lou van~den Dries, Angus Macintyre, and David Marker.
		\newblock {Logarithmic-exponential series}.
		\newblock \emph{Ann. Pure Appl. Log.}, 111\penalty0 (1-2):\penalty0 61--113,
		jul 2001.
		\newblock ISSN 01680072.
		\newblock \doi{10.1016/S0168-0072(01)00035-5}.
		
		\bibitem[van~den Dries et~al.(2019)van~den Dries, van~der Hoeven, and
		Kaplan]{VandenDries2018}
		Lou van~den Dries, Joris van~der Hoeven, and Elliot Kaplan.
		\newblock {Logarithmic hyperseries}.
		\newblock \emph{Trans. Am. Math. Soc.}, 372\penalty0 (7):\penalty0 5199--5241,
		jul 2019.
		\newblock \doi{10.1090/tran/7876}.
		
		\bibitem[Wilkie(1996)]{Wilkie1996}
		Alex~J. Wilkie.
		\newblock {Model completeness results for expansions of the ordered field of
			real numbers by restricted Pfaffian functions and the exponential function}.
		\newblock \emph{J. Amer. Math. Soc}, 9\penalty0 (4):\penalty0 1051--1094, 1996.
		
	\end{thebibliography}

\end{document}